\documentclass[11pt,reqno]{amsart}

\usepackage{amsmath}
\usepackage{dsfont}  
\usepackage{amsthm}
\usepackage{graphicx}
\usepackage{mathtools}
\usepackage{amsfonts}
\usepackage{amssymb}
\usepackage{hyperref}
\usepackage{enumerate}
\usepackage{enumitem}
\usepackage{bm}
\usepackage{comment}
\usepackage[margin=1 in]{geometry}
\usepackage[normalem]{ulem}
\usepackage{tikz}
\usepackage{tikz-cd}
\usepackage{xcolor}
\usetikzlibrary{matrix,arrows,positioning,automata}
  \usepackage{cancel}

\usepackage[numbers]{natbib}
\usepackage{todonotes}

\numberwithin{equation}{section}
\newtheorem{theorem}{Theorem}[section]
\newtheorem{lemma}[theorem]{Lemma}
\newtheorem{proposition}[theorem]{Proposition}

\newtheorem{assumption}[theorem]{Assumption}

\theoremstyle{definition}
\newtheorem{example}[theorem]{Example}\newtheorem{definition}[theorem]{Definition}
\newtheorem{remark}[theorem]{Remark}

\def\E{{\mathbb E}}

\def\R{{\mathbb R}}
\def\N{{\mathbb N}}

\def\P{{\mathcal P}}

\def\Q{{\mathcal Q}}

\def\L{{\mathcal L}}

\def\F{{\mathcal F}}

\def\red{\textcolor{red}}
\def\blue{\textcolor{blue}}

\DeclareMathOperator*{\argmin}{arg\,min}
\DeclareMathOperator*{\essinf}{ess\,inf}

\newcommand\cA{\mathcal A}

\newcommand\cE{\mathcal E}
\newcommand\cF{\mathcal F}
\newcommand\cG{\mathcal G}

\newcommand\cJ{\mathcal J}

\newcommand\cL{\mathcal L}

\newcommand\cN{\mathcal N}

\newcommand\cP{\mathcal P}

\newcommand\cT{\mathcal T}

\newcommand\cV{\mathcal V}
\newcommand\cW{\mathcal W}

\newcommand\cY{\mathcal Y}
\newcommand\cZ{\mathcal Z}

\def \E{\mathbb{E}}
\def \F{\mathbb{F}}

\def \L{\mathbb{L}}

\def \N{\mathbb{N}}

\def \P{\mathbb{P}}
\def \Q{\mathbb{Q}}
\def \R{\mathbb{R}}

\newcommand{\pa}[1]{\left(#1\right)}

\newcommand{\sqbra}[1]{\left[#1\right]}
\newcommand{\abs}[1]{\left|#1\right|}

\newcommand{\jo}[1]{{\textcolor{blue}{#1}}}

\date{\today}
\thanks{The first and second authors (in alphabetical order) acknowledge financial support by the Deutsche Forschungsgemeinschaft (DFG, German Research Foundation) -- SFB 1283/2 2021 -- 317210226.\ The third and fourth acknoweldge  financial support by NSF CAREER award DMS-2143861 and the AMS Claytor-Gilmer fellowship.}
 
\title[Pasting and Donsker-type results]{Pasting of equilibria and Donsker-type results for mean field games}
\author{Jodi Dianetti, Max Nendel, Ludovic Tangpi and Shichun Wang}
\address{University of Rome Tor Vergata}
\email{jodi.dianetti@gmail.com}
\address{Bielefeld University}
\email{max.nendel@uni-bielefeld.de}
\address{Princeton University}
\email{ludovic.tangpi@princeton.edu}
\email{shichun.wang@princeton.edu}

\begin{document}

\begin{abstract}
This paper studies the relation between equilibria in single-period, discrete-time and continuous-time mean field game models.
First, for single-period mean field games, we establish the existence of equilibria and then prove the propagation of the Lasry-Lions monotonicity to the optimal equilibrium value, as a function of the realization of the initial condition and its distribution.
Secondly, we prove a pasting property for equilibria; 
that is, we construct equilibria to multi-period discrete-time mean field games by recursively pasting the equilibria of suitably initialized single-period games.
Then, we show that any sequence of equilibria of discrete-time mean field games with discretized noise converges (up to a subsequence) to some equilibrium of the continuous-time mean field game as the mesh size of the discretization tends to zero.
When the cost functions of the game satisfy the Lasry-Lions monotonicity property, we strengthen this convergence result by providing a sharp convergence rate. 
\end{abstract}

\maketitle

\date{today}
\section{Introduction}

Mean field games (MFGs) were introduced independently in \cite{huang2006large, lasry2007mean} and nowadays represent the standard way to analyze symmetric games with a large number of players in weak interaction. 
We refer to the two-volume textbook \cite{cardelbook1, cardelbook2} and to the references therein for a detailed discussion on MFGs and their numerous applications.
Within the extensive literature on MFGs, a clear distinction can be made in terms of the time-set on which players' actions can take place.
This can be in discrete or continuous-time.
A dominant majority of works is concerned with continuous-time models. 
This setting lends itself well to powerful analytical methods based on (systems of) partial differential equations and (forward) backward stochastic differential equations (FBSDE).
We refer for instance to \cite{cardaliaguet2019master,cardelsicon,CarmonaLacker15,Car-Del-Lack16,delarue2020master,dianetti.ferrari.fischer.nendel.2022unifying,fischer2017connection,Gang-Mes-Mou-Zhang22,gomes.mohr.souza.2013,jackson.tangpi.2023quantitative,lacker.2016general,Mou-Zhang22} among many references.
On the other hand, discrete-time models are popular in many applications in economics (see e.g. \cite{Weintraubetal2008}) and are widely employed in the context of numerical simulation (see e.g. \cite{guo.hu.xu.zhamg.2019learning}).
In this setting, most works consider models with finite state space, see for instance \cite{Bon-Lavi-Pfei23,Huang12,saldi.basar.raginsky.2018markov,Saldi-Basa-Rag20,Saldi-Basar-Rag23,Wiecek24,Yang-Ye18} and base the analysis on techniques from Markov chain theory.

The present paper focuses on a general discrete-time mean field games with the goal of studying structural and asymptotic properties of discrete-time mean field equilbria (MFE).
We will develop two types of results along these lines: Pasting of discrete-time equilibria and convergence of discrete-time equilibria to continuous-time  analogues.
Pasting procedures (described below) extend the idea of Bellman's optimality principle or dynamic programming to differential games, and pave the way to numerically simulate discrete-time equilibria via iterated single-period games.
On the other hand, the importance of providing sufficient conditions for convergence to continuous-time equilibria is underlined by the general instability of Nash equilibria with respect to small perturbations of the data, which is a well-known phenomenon in the game-theoretic literature.
In the following, we describe the context and main contributions of the paper before discussing related literature.

\subsection{Summary of the main results of the paper} 

The data of the problem are given by functions  $b,L,G$ with suitable domains, and a constant volatility matrix $\sigma$. 
For a time horizon $T>0$ and $k \in \N$, the game is on $k$ periods of length $\delta = \frac Tk$, and we set 
$t_i := \delta i, \ i=0,\dots,k$.
The initial position is described by a random variable $\xi$ with law $m_\xi$, the idiosyncratic noise is a $\xi$-independent discrete-time stochastic process $Z^k=(Z^k_{t_i})_{i=0,...,k}$ {with independent stationary increments}.
In the $k$-period MFG, given a population distribution $m=(m_{t_i})_{i=0,\dots,k}$ with $m_{t_0} = m_\xi$, the representative player chooses a sequence of feedback maps $\alpha = (\alpha_{t_i}(\cdot))_{i=0,...,k}$ in order to minimize the cost
\begin{equation*}
\begin{array}{l}
    J^k_m(\alpha) :=\E\Big[\sum_{i=0}^{k-1}L(X^{\alpha}_{t_i},\alpha_{t_i}(X^{\alpha}_{t_i}), m_{t_i})\delta + G(X^{\alpha}_{t_k}, m_{t_k}) \Big], \\
    \text{subject to } X^\alpha_{t_{i}} = \xi + \sum_{j=0}^{i-1}b(X^\alpha_{t_j},\alpha_{t_j}(X^\alpha_{t_j}), m_{t_j})\delta +  Z^k_{t_{i}},\quad i=0,\dots,k. 
\end{array}
\end{equation*}
An MFE is a pair $(\alpha^k, m^k)$ such that $\alpha^k$ is a minimizer of $J^k_{m^k}$ and 
$\P\circ (X^{\alpha^k}_{t_i})^{-1} = m^k_{t_i}$ for all $i=0,\dots,k$.

\subsubsection{Pasting of discrete-time equilibria}
The first objective of the paper is to give conditions allowing to construct MFEs above as concatenation of single-period equilibria, hence giving further insight in the structure of equilibria.
For instance, it is well-known that the construction of (continuous-time) MFE is often based on monotonicity properties such as Lasry-Lions monotonicity \cite{Card17,LasryLions07} or displacement monotonicity \cite{ahuja2016wellposedness,Gang-Mes-Mou-Zhang22} (also see \cite{Graber-Mezaros23,Mou-Zhang22} for other monotonicity conditions).
However, these methods break down for instance if the coefficients are time-dependent and satisfy different types of monotonicity properties on different intervals.
In this context, pasting results allow to construct equilibria on sub-intervals, where monotonicity holds, and paste them together to recover a global equilibrium.

Our first main result, Theorem \ref{thm.rep.full}, shows that an MFE of the $k$-period problem can be constructed as concatenation of single-period equilibria of games with appropriate terminal costs.
These terminal costs can be thought of as ``value functions'' of single-period games and naturally arise from the dynamic programming principle.
Our construction requires each single-period game (i.e., when $k=1$ in the game described above) to be uniquely solvable.
To this end, we carefully revisit well-posedness of the single-period MFGs by showing existence of  equilibria with \emph{Markovian} controls, and uniqueness is guaranteed by Lasry-Lions monotonicity, see Theorem \ref{thm:exits.one.period} and Proposition \ref{prop:unique.single.period}.
An essential step in our approach is the derivation of a \emph{propagation of monotonicity} from the cost functions to the value of the problem.
This is reminiscent of the propagation of displacement monotonicity recently proved in \cite{gangbo.meszaros.mou.zhang.2022} using the master equation.
The non-uniqueness regime is certainly interesting, but beyond the scope of the current paper.

\subsubsection{Donsker-type results}
Another heuristic description of the result discussed above is that simple single-period MFGs form the basic building block allowing to construct more complex multi-period equilbria.
It is natural to inquire whether this principle extends to general continuous-time equilibria.
This type of question has a long history in probability theory with its most celebrated manifestation being Donsker's construction of the Brownian motion.
The second objective of this work is to recover similar results for continuous-time MFGs driven by Brownian motions.
Here, the continuous-time game takes the following form:
We fix as idiosyncratic noise a $\xi$-independent Brownian motion $W=(W_t)_{t\in [0,T]}$.
For a continuous-time flow of measures $(m_t)_{t\in [0,T]}$, the representative player optimizes over feedback functions $\alpha = (\alpha_t(\cdot))_{t\in [0,T]}$ the cost
\begin{equation*}
\begin{array}{l}
    J_m(\alpha) :=\E\Big[\int_0^TL(X^\alpha_s,\alpha_{s} (X^\alpha_s), m_{s})ds + G(X^\alpha_{T}, m_{T}) \Big], \\
    \text{subject to } X^\alpha_{t} = \xi + \int_0^tb(X^\alpha_{s},\alpha_{s}(X^\alpha_s), m_{s})ds + \sigma W_t ,\quad t \in [0,T]. 
\end{array}
\end{equation*}
The pair $(\alpha, m)$ is an MFE for the continuous-time MFG if $\alpha$ is optimal for $J_m$ and $\P\circ(X^\alpha_t)^{-1} = m_t$ for all $t\in [0,T]$.
In this setting, we show that when the mesh size $\delta$ of the discretization of the time interval $[0,T]$ converges to zero, then any discrete-time equilibria form a tight sequence with any subsequential limit being an equilibrium for the associated continuous-time game.
We thus obtain a sort of \emph{invariance principle for MFGs}.
This requires the discrete-time noise $Z^k$ to converge (in distribution) to $W$.
For instance, the case where $Z^k$ is a sequence of appropriately scaled binomial random variables can be considered, see Theorem \ref{thm:compactness}.

In the special case of games on Gaussian processes (e.g., discretized Brownian motions) and when the cost functions are Lasry-Lions monotone, the above result can be made a lot stronger.
In this case, both the discrete-time and continuous-time games are uniquely solvable; and additional regularity and convexity properties allow to derive a sharp, non-asymptotic rate for the convergence of the discrete-time equilibrium to its continuous-time analogue, see Theorem \ref{thm:BSDE.rate.LL}.
This result builds upon prior developments on backward stochastic difference equations (BS$\Delta$E) as studied for instance in \cite{Brian-Dely-Mem01,Brian-Dely-Mem02,Cher-Stad13,Coh-Elli12} among many others.
These papers study well-posedness of such equations, as well as other crucial properties such as comparison, convex dual representation and robustness.

In addition to presenting new structural properties of MFEs, we believe the results of this paper to be potentially relevant to design efficient numerical simulations of general games.
In fact, the BS$\Delta$E theory allows to reduce the problem of computing a discrete-time MFE to that of iteratively solving BS$\Delta$Es, a problem that is amenable to efficient simulation techniques and for which various methods have been proposed, see e.g. \cite{beck2023overview,bouchard2004discrete,crisan2014second,gobet2005regression,hutzenthaler2020proof,zhang2004numerical}.
We describe this approach in subsection \ref{sec:simulation} without empirical studies as numerical simulations are not the focus of the present work, but recognizing that this demands careful investigation.

\subsection{Related literature} 

There is a rich literature on discrete-time MFGs.
The vast majority of papers on the subject focusing on finite state spaces.
Discrete-time models were first studied in the economic literature \cite{Weintraubetal2008} under the notion of oblivious equilibrium for infinite models.
One of the first theoretical papers on discrete-time MFGs is  \cite{gomes.mohr.souza.2010}.
In this paper the authors focus on existence of stationary solutions of the game and investigate exponential convergence to this stationary solution using an entropy regularization technique.
For subsequent developments on discrete-time  games we refer to \cite{Bon-Lavi-Pfei23,
guo.hu.xu.zhamg.2019learning,
Huang12,
saldi.basar.raginsky.2018markov,
Saldi-Basa-Rag20,
Saldi-Basar-Rag23,
Wiecek24,
Yang-Ye18,
lauriere.Perrin.Geist.Pietquin.2022learning}.
Some references dealing with arbitrary state spaces can be found in \cite{Bonn-Lavi-Pfe21,Car-coo-gra-lau22,Lacker-Ramn19}.
Most of these papers are concerned with solvability issues and applications of reinforcement learning techniques to competitive games.
While we prove well-posedness results in the present paper, this is not our focus.
We rather present structural properties of discrete-time equilibria.
Let us also remark that our results on discrete-time games cover finite state spaces, as the noise sources $\xi$ and $Z^k$ are allowed to be discrete random variables.
To the best of our knowledge, pasting of MFE has been investigated only in \cite{mou.zhang.2024minimal}. 
The approach in this paper is based on the master equation.

While convergence of symmetric $N$-player stochastic differential games to mean field games as $N\to\infty$ has been the subject of intensive research, cf.\ \cite{cardaliaguet2019master,delarue2020master,Djete23AAP,fischer2017connection,jackson.tangpi.2023quantitative,lacker2016general,lacker2020convergence,laurieretangpi,Poss-Tangp24}, a lot less is known about convergence of discrete-time games to their continuous-time analogues.
The paper closest to addressing this issue is
\cite{achdou.camilli.dolcetta.2013mean}, where the convergence of the finite-difference method for the mean field game system of partial differential equations is obtained.
Note that, despite numerical methods typically going through some discretization of the equations related to the MFG problem as in \cite{achdou.camilli.dolcetta.2013mean},  no theoretical interpretation of these discretized equations in terms of discrete-time games is discussed in the literature, which is a question addressed in the current paper. 
Relatedly, the stability of \emph{continuous-time} finite-state MFG master equations towards their continuous-state versions has been studied in the recent paper \cite{bertucci.cecchin.2024mean}, 
while  \cite{hadikhanloo.silva.2019} investigates the convergence of discrete-time MFE with respect to discretizations of the state space in the case of deterministic dynamics with no idiosyncratic noise. 

Closely related to the topic of numerical methods is the problem of providing ``simpler'' problems allowing to theoretically approximate MFE. 
This was first addressed in \cite{CardaliaguetHadikhanloo17}, showing the convergence of the \emph{fictitious play} algorithm in potential MFGs through PDE methods, see also \cite{elie2019approximate, perrin2020fictitious, xie2020provable}. 
While the convergence of algorithms is often shown under the uniqueness of the equilibrium, 
for \emph{submodular} MFGs, it is possible to approximate minimal and maximal MFE via a direct iteration of the best-response map, cf.\ \cite{dianetti.ferrari.fischer.nendel.2019, dianetti.ferrari.fischer.nendel.2022unifying}, and via the fictitious play, cf.\ \cite{dianetti2022strong}, without relying on the uniqueness of the equilibrium, cf.\ \cite{dianetti.ferrari.federico.floccari.2024multiple}.

\subsection{Organization of the paper}
In  Section \ref{section:single period}, we study the properties of single-period MFGs while, in  Section \ref{sec:multip}, we paste the single-period equilibria to construct multi-period equilibria.
Section \ref{sec:convergence to continuous} studies the convergence of the discretized MGgs to continuous-time MFGs.
The related BS$\Delta$Es are introduced and discussed in Section \ref{sec:Donsker}, together with Donsker-type results and the sharp rate of convergence.

\subsection{Notation}
In this section, we introduce some frequently used notation in the paper.
{Let $d \in \mathbb{N}, d>0$ be fixed, and for $x,y \in \R^d$ denote by  $xy$  the scalar product in $\R^d$.
We use $|\cdot|$ to denote both the absolute value and the Euclidean norm on $\R^d$.
For generic $l,l_1, l_2\in \mathbb{N}, l,l_1, l_2>0$ and  matrices $\sigma_1 \in \R^{l_1 \times l},\sigma_2 \in \R^{l \times l_2}$, $\sigma_1 \sigma_2$ denotes the  matrix product, while  $\sigma_1^{\text{\tiny {$\top$}}}$ indicates the transpose of $\sigma_1$.}
For any complete separable metric space $E$, and any $p\ge1$, we denote by $\cP_p(E)$ the set of Borel probability measures $\mu$ on $E$ with finite $p^{\mathrm{th}}$ moment, and we set $\| \mu \|_p := \big( \int_E |x|^p \mu(dx) \big)^{1/p}$, when $E$ is a subset of $\R^d$.
We equip this set with the Wasserstein distance of order $p$ denoted by
\begin{equation*}
  \cW_p(\mu,\nu):=\bigg(\inf_{\pi}\int_{E\times E}d_E(x,y)\pi(dx,dy)\bigg)^{1/p},
\end{equation*}
where $d_E$ is the metric on $E$ and the infimum is taken over all couplings of $\mu$ and $\nu$.
On a complete probability space $(\Omega,\cF,\P)$, we denote by $\L^p(\Omega,\cF,\P)$ the space of all (equivalence classes of) random variables with finite $p^{\mathrm{th}}$ moment and the expectation operator on this space is denoted by $\E^{\P}[\cdot]$.
When there is no potential for ambiguity, we simply write $\L^p$ and $\E[\cdot]$.\
Given a random variable $\xi\colon \Omega\to E$, we denote by 
$m_\xi$ the probability distribution of $\xi$ on $E$.

Throughout, let  $p\in (1,\infty)$ be fixed.
The main protagonists of the paper are an action set $A\subseteq \R^d$ with $0\in A$, which is assumed to be closed and convex, and functions $b$, $G$, and $L$ of the form
\begin{align*}
  &b\colon\R^d\times A\times \cP_p(\R^d)\to \R^d,\\
  &L\colon\R^d\times A\times \cP_p(\R^d)\to \R,\\
  &G\colon\R^d\times \cP_p(\R^d)\to \R.
\end{align*}
In our analysis, we can easily allow the functions $b$ and $L$ to depend on time and possibly have different dimensions in their first and second component.\
However, we do not make these assumptions to ease the presentation.

\section{Single-period MFGs}
\label{section:single period}
This section is devoted to the study of single-period MFGs (also known as static MFGs).\
They will form the main building block allowing to study more general cases.
Single-period MFGs have been only sparsely discussed in the literature (at least for games on general state spaces).\
Some references include \cite{Card-lect13} and \cite{Car-coo-gra-lau22}.\
These references consider rather simple models, in particular, they assume controls to take values in a compact action space and find equilibria in relaxed (or randomized) controls.\
A more recent work is \cite{Ryan23}, which discusses the approximation of static Nash equilibria and MFGs by continuous-time flows.\
In the sequel, we describe the single-period game, we consider, before stating the main result of the section.

Let $\delta>0$\footnote{The case $\delta=0$ is trivial. It does not restrict the generality to assume that $\delta>0$.}  and $(\Omega, \cF, \P)$  be a probability space carrying two $d$-dimensional independent random variables $\xi$ and $Z$ as well as an independent uniformly distributed random variable.
We denote by $\cF_0$ the $\P$-completion of the sigma algebra generated by $\xi$.
We assume $\xi\in \L^p(\Omega,\cF_0,\P)$ and $Z\in \L^p(\Omega,\cF,\P)$.
The set of admissible controls in the single-period game is denoted by $\cA_S(\xi)$ and contains all $\cF_0$-measurable and $p$-integrable random variables $\alpha$ with values in $A$.  

For most of the applications, we have in mind that the random variable $Z$ follows a Gaussian distribution with variance $\sigma\sigma^\top \delta$ for some $d\times d$ matrix $\sigma$.\ We will specialize the discussion to this case later when passing to the continuous-time setting.\ For now, we allow $Z$ to be fairly arbitrary, in particular, we do not exclude the cases of $Z$ being discrete, including $Z\equiv 0$, or $Z$ being a degenerate Gaussian, so that our model also allows for possibly discrete (including finitely many) values of the state, for instance, if $Z$ and $\xi$ follow Bernoulli or binomial distributions.

Given any $\alpha\in \cA_S(\xi)$, the controlled state variable is of the form
\begin{equation*}
  X^{\alpha} := \xi + b(\xi, \alpha, m_\xi)\delta +  Z.
\end{equation*}
We understand $X^\alpha$ as the value after one period of a state process started at time $0$ in the position $\xi$.\ In the single-period mean field game, we fix a probability measure $m\in \cP_p(\R^d)$ to represent the distribution of the population's state.\ The representative agent then chooses a strategy $\alpha\in \cA_S(\xi)$ that minimizes the cost $J_m$, given by
\begin{equation}\label{eq.def.cost}
    J_m(\alpha):=\E\big[L(\xi,\alpha, m_\xi)\delta + G(X^\alpha, m) \big]\quad\text{for }\alpha\in \cA_S(\xi).
\end{equation}
A mean field equilibrium (MFE) is a pair $(\alpha,m)$ such that $\alpha\in \cA_S(\xi)$ is a minimizer of $J_m$ and 
\begin{equation*}
  \P\circ (X^\alpha)^{-1} = m.
 \end{equation*} 
 In the remainder of the paper, we will refer to this problem as \emph{the single-period mean field game with initial position $\xi$ and terminal cost $G$}.
 This problem is fully determined by the data $(\xi,Z,b, L, G)$.\ Under our standing assumptions stated below, for each $m$, the cost function $J_m$ takes values in $\R\cup\{+\infty\}$ and has nonempty domain. 

\subsection{Existence and uniqueness}\label{sec:single:exists}

In addition to our standing assumptions, we assume the following.

\begin{assumption}(Single-period existence)
    \label{Ass:single.exists}
    Let $q\in [1,p]$ be fixed.
    \begin{enumerate}[label={(\roman*)}]
        \item \label{Ass:Single.ii} The function $b$ is continuous on $\R^d\times A\times \cP_p(\R^d,K)$, where the set $\cP_p(\R^d,K):=\big\{ m\in \cP_p(\R^d)\,\big|\, \int_{\R^d}|z|^p m(dz)\leq K\big\}$ is endowed with the subspace topology of the $\cW_1$-topology and $K$ is given in \eqref{eq:definition.K} below.
        Moreover, $b$ is of linear growth, i.e.
        \begin{equation*}
            |b(x,a,\mu)|\le C_b\big(1 + |x| + |a| + \|\mu\|_1\big)
        \end{equation*}
        for all $(x,a,\mu)\in \R^d\times A\times \cP_p(\R^d)$ and some constant $C_b\ge1$.
        \item \label{Ass:Single.iii} 
        The functions $L$ and $G$ are continuous on $\R^d\times A\times \cP_p(\R^d,K)$ and $\R^d\times \cP_p(\R^d,K)$, respectively.
        Moreover, $L$ and $G$ satisfy the following growth properties:
        \begin{gather*}
            -C_G\le G(x,\mu) \le C_G(1 +|x|^q),\\
            c_L(|a|^p-1)\le L(x,a,\mu)\,\text{ and } L(x, 0,\mu) \le C_L\Big(1 + |x|^q  + \|\mu\|_q^q \Big)
        \end{gather*}
        for all $(x,a,\mu)\in \R^d\times A\times \cP_p(\R^d)$ and for some constants $C_G,C_L\ge0$, and $c_L>0$.
        \item \label{Ass:Single.iv} For all $x\in \R^d$, $m\in \cP_p(\R^d)$ and $\mu\in \cP_p(\R^d)$, there exists a unique minimizer of the static optimization problem $\inf_{a\in A}\big(L(x,a,\mu)\delta + \E[G(x + b(x,a,\mu)\delta + Z,m)]\big)$. 
    \end{enumerate}
\end{assumption}
The constant $K$ is defined as
    \begin{equation}
    \label{eq:definition.K}
        K:=4^p\bigg(2+\E[|\xi|^p] + \frac{1}{c_L\delta}(C_J + C_G)  + \E[|Z|^p] \bigg)
\end{equation}
with    \begin{equation*}
        C_J := 2\max(C_L, 16^{q}C_b\vee 1)\Big(1+ \E[|\xi|^q]  +\E[|Z|^q]\Big) .
    \end{equation*}
Assumption \ref{Ass:single.exists} consists of standard growth conditions and mild regularity assumptions on the coefficients.\
The main condition of interest is the lower bound on $L$, which allows to get the coercivity property
\begin{equation*}
    \liminf_{|a|\to\infty} \frac{L(x,a,m)}{|a|} = +\infty
\end{equation*}
uniformly in $(x,m)\in \R^d\times \cP_q(\R^d)$.\ This condition can be dropped if the controls take values in a compact action space as sometimes assumed in the literature.\

\begin{remark}
    Observe that if $b$ is linear in $a$ and  the functions $G$ and $L$ are convex, one of them being strictly convex, then condition \ref{Ass:Single.iv} is satisfied. 
\end{remark}

\begin{theorem}
\label{thm:exits.one.period}
  Let Assumption \ref{Ass:single.exists} be satisfied.\ Then, the single-period mean field game with data $(\xi,Z, b, L,G)$ admits a mean field equilibrium $(\alpha,m)$.
\end{theorem}
The proof of Theorem \ref{thm:exits.one.period} builds on the following lemma:
\begin{lemma}
    \label{lem:optimization}
    Let Assumption \ref{Ass:single.exists} be satisfied and consider the function
        \[
        \cJ(x,\mu,m):=\inf_{a\in A} \bigg(L(x,a,\mu)\delta+\int_{\R^d} G\big(x+b(x,a,\mu)\delta+ z,m\big)m_Z(d z)\bigg)
    \]
    for $(x,\mu,m)\in \R^d\times \cP_p(\R^d)^2$. Then, there exists a unique function $a^*\colon \R^d\times \cP_p(\R^d)^2\to A$ such that for each $(x, \mu,m)\in \R^d\times \cP_p(\R^d)^2$
    it holds
 \begin{equation}\label{eq.lem.single-opt}
    \cJ(x,\mu,m)=L\big(x,a^*(x,\mu,m),\mu\big)\delta+\int_{\R^d} G\Big(x+b\big(x,a^*(x,\mu,m),\mu\big)\delta + z,m\Big)m_Z(d z).
 \end{equation}  
 The function $a^*$ is continuous on $\R^d\times \cP_p(\R^d,K)^2 $.
    \end{lemma}
\begin{proof}
      Let $(x, \mu,m)\in \R^d\times \cP_p(\R^d)^2$ be given and let $a^n\in A$ be a minimizing sequence for $\cJ(x, \mu,m)$.
    That is, a sequence such that
    \[
        \cJ(x,\mu,m)\ge \bigg(L(x,a^n,\mu)\delta+\int_{\R^d} G\big(x+b(x,a^n,\mu)\delta+ z,m\big)m_Z(d z)\bigg)  -\frac1n. 
    \]
  Thus, for all $n\in \N$ we have
  \begin{align*}
    L(x, a^n,\mu)\delta &\le \cJ(x, \mu,m) + \frac1n - \int_{\R^d} G\big(x+b(x,a^n,\mu)\delta+ z,m\big)m_Z(d z) \\
    &\le \cJ(x, \mu,m) + \frac1n +C_G.
  \end{align*}
  Using the growth conditions on $b,L$ and $G$ as well as $m\in \cP_p(\R^d)$, it follows that
  \begin{equation}
  \label{eq:L.bounded}
    \cJ( x,\mu, m)\le C_J(x)<\infty
  \end{equation}
  with a constant $C_J(x)$ which does not depend on $m$.
    Indeed, we have
    \[
        \cJ(x,\mu,m) \leq L(x,0,\mu)\delta +C_G\bigg(1+\int_{\R^d}\big|x +b\big(x,0,\mu\big)\delta+ z\big|^qm_Z(d z)\bigg).
    \]
    Moreover, by Assumption \ref{Ass:single.exists},
    \begin{align*}
        \int_{\R^d} \big|x+b\big(x,0,\mu\big)\delta + z\big|^q\,m_Z(d z)&\leq (16^{q-1}C_b\vee 1)\big(1+ |x|^q + \|\mu\|_q^q + \|m_Z\|_q^q\big) 
    \end{align*}
    and
    \[
     L(x,0,\mu)\leq C_L\big(1 + |x|^q  + \|\mu\|^q_q \big).
    \]
    Thus we can take
    \begin{equation}
    \label{eq:Def.CJ}
        C_J(x) = \max(C_L, 16^{q}C_b\vee 1)\big(1+ |x|^q + \|\mu\|_q^q + \|m_Z\|^q_q\big) .
    \end{equation}
  Thus, there is a constant $C>0$ 
  such that
  \begin{equation*}
  \label{eq:alphan.L-bounded}
    L(x, a^n,\mu)\le C \quad \text{for all } n\in \N.
  \end{equation*}
  Therefore, by Condition \ref{Ass:single.exists}.$(ii)$ there is $a^*(x,\mu,m) \in A$ such that up to a subsequence $(a^n)_{n\ge1}$ converges  to $a^*(x,\mu,m)$. 
  Hence, using  lower semicontinuity of $G$ and $L$ 
  we have
  \begin{align*}
    \cJ(x,\mu,m) &\ge \liminf_{n\to\infty}\bigg(L(x,a^n,\mu)\delta+\int_{\R^d} G\big(x+b(x,a^n,\mu)\delta+ z,m\big)m_Z(d z)\bigg)\\
    &\ge \bigg(L(x,a^*(x,\mu,m),\mu)\delta+\int_{\R^d} G\big(x+b(x,a^*(x,\mu,m),\mu)\delta+ z,m\big)m_Z(d z)\bigg),
   \end{align*} 
  where we also used Fatou's lemma.
  This shows that $a^*(x,\mu,m)$ is optimal.
  Uniqueness follows by Condition \ref{Ass:single.exists}.$(iii)$.

  \medskip

  We now show that the mapping $a^*\colon \R^d\times \cP_p(\R^d)^2\to A$ is continuous on $\R^d\times \cP_p(\R^d,K)^2$.
    Let $(x^n,\mu^n,m^n)_{n\geq 1}\subset \R^d\times \cP_p(\R^d,K)^2$ be an arbitrary sequence that converges to  $(x,\mu,m)\in \R^d\times \cP_p(\R^d,K)^2$. Since the sequence $(x_n)_{n\in \N}$ converges, it is bounded by a constant $c\geq0$. Hence, by \eqref{eq:L.bounded},
    \[
     L\big(x^n,a^*(x^n,\mu^n,m^n),\mu^n\big)\delta\leq  C_J(c) + C_G
     \quad\text{for all }n\in \N.
    \]
    By Assumption \ref{Ass:single.exists}, it follows that
    \[
    | a^*(x^n,\mu^n,m^n)|^p\leq 1+ \frac{C_J(c) + C_G}{c_L}
    \quad\text{for all }n\in \N.
    \]
    Since $A$ is closed, by potentially passing to a subsequence, we may w.l.o.g.\ assume that there exists some $a^*\in A$ with $ a^*(x^n,\mu^n,m^n)\to a^*$ as $n\to \infty$.
    Using the continuity of $(b,L)$ and $G$ on $\R^d\times A\times \cP_p(\R^d,K)$ and $\R^d\times \cP_p(\R^d,K)$, respectively, it follows by dominated convergence that
    \begin{align*}
        L(x,a^*,\mu)\delta+\int_{\R^d} G\big(x+b(x,a^*,&\mu) \delta +z,m\big)m_Z(d z)= \lim_{n\to \infty} \Big( L\big(x,a^*(x^n,\mu^n,m^n),\mu^n\big)\delta\\
        &\quad+\int_{\R^d} G\Big(x^n+b\big(x^n,a^*(x^n,\mu^n,m^n),\mu^n\big)\delta + z,m^n\Big)m_Z(d z) \Big) \\
    &\leq  \lim_{n\to \infty} L(x^n,a,\mu^n)\delta+\int_{\R^d} G\big(x^n+b(x^n,a,\mu^n)\delta + z,m^n\big)m_Z(d z)\\
    &= L(x,a,\mu)\delta+\int_{\R^d} G\big(x+b(x,a,\mu)\delta+ z,m\big)m_Z(d z)\quad\text{for all }a\in A.
    \end{align*}
    Hence by uniqueness of the minimizer $a^*=a^*(x,\mu,m)$, and the continuity follows since we have shown that every subsequence has a further subsequence that converges to the same limit $a^*(x,\mu,m)$.
\end{proof}

\begin{proof}[Proof of Theorem \ref{thm:exits.one.period}]
   The existence of a mean field equilibrium will use (a variant of) the Kakutani fixed point theorem.
   We split the proof into two main steps.
   In the first one, we show that the control problem is well-posed for each measure $m$; and in the second and third step, we show that this well-posedness allows to obtain a mapping whose fixed points exist and are mean field equilibria.

   \emph{Step 1: construction of the fixed point mapping.}
  In this step, we show that for each $m \in \cP_p(\R^d)$, the control problem with objective function $J_m$ admits a unique optimal control $\alpha$, which will allow to construct a function $\Psi$ mapping $\cP_p(\R^d)$ into itself as $\Psi(m) = m_{X^\alpha}$.

  Let $m \in \cP_p(\R^d)$ be given.
  By Lemma \ref{lem:optimization}, it follows that
        \begin{align*}
            J_m\big(a^*(\xi,m_\xi,m)\big)\leq \E\Big[L\big(\xi,\bar\alpha(\xi),m_\xi\big)\delta+ G(X^{\bar\alpha},m)\Big]
        \end{align*}    
    for all measurable maps $\bar\alpha\colon \R^d\to A$. 
    This shows that $\alpha := a^*(\xi,m_\xi,m)$ is optimal. 
    This optimal control is $\P$-a.s. unique in $\cA_S(\xi)$.
    In fact, let $\beta\colon \Omega\to A$ be $\sigma(\xi)$-measurable and such that
    \[
        J_m(\beta)\leq \inf_{\alpha\in \cA_S(\xi)} J_m(\alpha).
    \]
    Assume towards a contradiction that $\P(N)>0$ for $N:=\{\beta\neq a^*(\xi,m_\xi,m)\}$, then by Assumption \ref{Ass:single.exists},
        \begin{align*}
            \E\big[\big( L(\xi,&\beta, m_\xi)\delta + G(X^\beta, m)\big)1_{N}\big]\\
            &=\E\bigg[\bigg( L(\xi,\beta, m_\xi)\delta + \int_{\R^d}G\big( \xi+b(\xi,\beta,m_\xi)\delta + z, m)\big)m_Z(dz)\bigg)1_{N}\bigg]\\
        &>\E\bigg[\bigg( L\big(\xi,a^*(\xi,m_\xi,m), m_\xi\big)\delta + \int_{\R^d}G\Big( \xi+b\big(\xi,a^*(\xi,m_\xi,m),m_\xi\big)\delta+z, m)\Big)m_Z(dz)\bigg)1_{N}\bigg]\\
        &=\E\Big[\Big( L\big(\xi,a^*(\xi,m_\xi,m), m_\xi\big)\delta + G\big(X^{a^*(\xi,m_\xi,m)}, m\big)\Big)1_{N}\Big].
        \end{align*}
        Hence,
        \[
        J_m(\beta)>J_m\big(a^*(\xi,m_\xi,m)\big)=\inf_{\alpha\in \cA_S(\xi)} J_m(\alpha) \ge J_m(\beta),
        \]
        which is a contradiction.
        Therefore, $\beta=a^*(\xi,m_\xi,m)$.

        We have thus shown that the following mapping is single-valued and well-defined:
   \begin{equation*}
    \Psi(m) := m_{X^\alpha}, \quad \text{with}\quad \alpha \quad \text{optimal for } J_m.
   \end{equation*}
   In order to conclude the proof, we will show that $\Psi$ admits a fixed point.
  It is immediate that the set $\cP_p(\R^d,K)$ is convex and $\cW_1$-compact.

\medskip

  \emph{Step 2: A priori estimation.}
  In this step, we construct $K>0$ such that if $\alpha$ is optimal for $J_m$, then $\int_{\R^d}|z|^pm(dz)\le K$. 
  Using the growth condition on $b$, for given $p \geq 1$ we have
\begin{equation*}
  \E[|X^\alpha|^p] \le 4^p(1+\E[|\xi|^p] + \E[|\alpha|^p] + \E[|Z|^p]).
\end{equation*}

On the other hand, by \eqref{eq:L.bounded} and optimality of $\alpha$ for $J_m$, we have by the growth condition on $L$ that
\begin{align*}
  c_L(\E[|\alpha|^p] - 1) &\le \E[L(\xi,\alpha,m_\xi)]\\
  & =\frac{1}{\delta}\Big(\inf_{\beta \in \cA_S}J_m(\beta) - \E[G(X^\alpha,m)]\Big)\\
  &\le \frac{1}{\delta}(C_J + C_G)
\end{align*}
with $C_J:= \E[C_J(\xi)]$ which is finite since $\xi$ has finite $q$-moments.
Thus,
\begin{equation*}
  \E[|X^\alpha|^p] \le 4^p\bigg(2+\E[|\xi|^p] + \frac{1}{c_L\delta}(C_J + C_G)  + \E[|Z|^p] \bigg) =K.
\end{equation*}

   \emph{Step 3: Existence of fixed points.}
 First observe that the mapping $\Psi$ is continuous.
  Let $(m^n)_{n\ge1}$ in $\cP_p(\R^d,K)$ be a sequence converging in $\cW_1$ to $m$.
  We know that for all $n\ge1$, $\Psi(m^n) = m_{X^{\alpha^n}}$
    where $\alpha^n=a^*(\xi, m_\xi,m^n)$.
    Since $a^*$ and $b$ are continuous, it follows that $(\alpha^n)_{n\ge1}$ and $(X^{\alpha^n})_{n\ge1}$ respectively converge to $\alpha$ and $X^{\alpha}$ $\P$-almost surely. 
    This also shows that $(m_{X^{\alpha^n}})_{n\ge1}$ converges to $m_{X^\alpha}$ in $\cW_1$.
    Thus, for every $\beta \in \cA_S(\xi)$, we have
   \begin{align*}
      \E\big[L(\xi, \beta, m_\xi)\delta +G(X^\beta, m^n) \big] &\ge J_{m^n}(\alpha^n) = \E\big[L(\xi, \alpha^n, m_\xi)\delta + G(X^{\alpha^n}, m^n)\big],
   \end{align*}
   which shows that $\alpha$ is optimal.
   That is, $\Psi(m) = m_{X^\alpha}$.

  Moreover, $\Psi$ maps $\cP_p(\R^d)$ into $\cP_p(\R^d,K)$.
  This easily follows from Step 2. 
   In fact, we showed in that step that if $\alpha$ is optimal then $\E[|X^\alpha|^p] \le K$.
   Thus, $\Psi(m)\in  \cP_p(\R^d,K)$. 
  We therefore conclude by the Brouwer-Schauder-Tychonoff fixed point theorem that the mapping $\Psi$ admits at least one fixed point.
  \end{proof}

We now discuss uniqueness of the MFE.\ To this end, we recall the standard Lasry-Lions monotonicity condition.
\begin{definition}
\label{Def:LL}
  A function $U\colon \R^d\times \cP_p(\R^d)\to \R$  is said to satisfy the {Lasry-Lions monotonicity} if
  \begin{equation*}
    \int_{\R^d}\big(U(x,m^1) - U(x,m^2)\big)(m^1 - m^2)(dx)\ge 0,
  \end{equation*}
   for all $m^1,m^2\in \cP_p(\R^d)$.
\end{definition}

\begin{proposition}
\label{prop:unique.single.period}
  If the terminal cost $G$ is Lasry-Lions monotone, then the single-period MFG admits at most one MFE.
\end{proposition}
\begin{proof}
  Assume to the contrary that there exist two distinct MFE $(m^1,\alpha^1)$ and $(m^2,\alpha^2)$.
  Clearly, $\P(\alpha^1\neq\alpha^2)>0$ since otherwise $\alpha^1 = \alpha^2$ $\P$-a.s., and thus $X^{\alpha^1} = X^{\alpha^2}$, which implies $m^1 = m^2$, yielding $(m^1,\alpha^1) = (m^2,\alpha^2)$.\
  Since $\P(\alpha^1\neq\alpha^2)>0$, it follows that $J_{m^1}(\alpha^1)<J_{m^1}(\alpha^2)$ and $J_{m^2}(\alpha^2)<J_{m^2}(\alpha^2)$.
  Therefore, we have 
  \begin{align*}
    0 &> J_{m^1}(\alpha^1) - J_{m^1}(\alpha^2) + J_{m^2}(\alpha^2) - J_{m^2}(\alpha^1)\\
    & = J_{m^1}(\alpha^1) - J_{m^2}(\alpha^1) - \big( J_{m^1}(\alpha^2) - J_{m^2}(\alpha^2)\big)\\
    & = \E\big[L(\xi, \alpha^1,m_\xi)\delta + G(X^{\alpha^1},m^1)\big] - \E\big[L(\xi, \alpha^1,m_\xi)\delta + G(X^{\alpha^1},m^2) \big]\\
    &\quad - \big( \E\big[L(\xi, \alpha^2,m_\xi)\delta + G(X^{\alpha^2},m^1)\big] - \E\big[L(\xi, \alpha^2,m_\xi)\delta + G(X^{\alpha^2},m^2)\big] \big)\\
    & = \E[ G(X^{\alpha^1},m^1)- G(X^{\alpha^1},m^2)] - \big( \E[ G(X^{\alpha^2},m^1)- G(X^{\alpha^2},m^2)] \big),
  \end{align*}
  which contradicts the fact that $G$ satisfies Lasry-Lions monotonicity condition.
\end{proof}

\subsection{Properties of the value function: propagation of Lasry-Lions monotonicity}
    \label{subsection Properties of the value function under LL monotonicity}

In this subsection, we discuss crucial properties of the value function of the game.
To this end, we introduce additional assumptions and notation.

\begin{assumption}
\label{Ass:multi}\
In addition to Assumption \ref{Ass:single.exists}, we assume the following.
\begin{enumerate}[label={(\roman*)}]
  \item  \label{Ass:multi.i} $p>q$ and 
  \[
   L(x,0,m)\leq c_L\big(1+|x|^q\big)\quad\text{for all }x\in \R^d\text{ and }m\in \cP_p(\R ^d).
  \]
  \item \label{Ass:multi.ii} The function $G$ satisfies the Lasry-Lions monotonicity condition.
\end{enumerate}
\end{assumption}

Then, the results of the previous subsection show that, for any $p$-integrable initial state $\xi$, the MFG admits an MFE $(\alpha,m)$ which is unique. Moreover, $m$ depends only on the law of $\xi$.\footnote{This is clear from the fact that the MFE is the fixed point of the mapping $\Psi$ introduced in the proof of Theorem \ref{thm:exits.one.period} and this $\Psi$ depends only on the law of $\xi$.}
We can then define a mapping
\begin{equation*}
  (\mu,G) \mapsto m^{\mu,G},
 \end{equation*} 
 which assigns the unique equilibrium law $m^{\mu,G}$ to any (initial) law $\mu\in \cP_p(\R^d)$ and any function $G$ satisfying Assumption \ref{Ass:multi}.\
We then consider the function
\begin{equation}
\label{eq:def.calG}
  \mathcal G(G, x, \mu) := \inf_{a\in A}\bigg(L(x,a, \mu)\delta+ \E\Big[G\big(X^{x,a}, m^{\mu,G}\big) \Big]\bigg)
\end{equation}
with
\begin{equation*}
  X^{x,a} := x + b(x,a, \mu)\delta + Z.
\end{equation*}
We call $\mathcal G$ the value function of the game associated with the costs $L$ and $G$ and an initial state with law $\mu\in \cP_p(\R^d)$.
We do not explicitly indicate the dependence of $\mathcal G$ on $L$ since, in the next section, we only need to vary $G$ while keeping $L$ fixed.\
The aim of the section is to derive several useful properties of the function $\cG$.
We start with the following remark. 
\begin{remark}
\label{rem:dependence.G}
  The function $\cG$ does not depend on the realizations of the random variables $\xi$ and $Z$, but on their respective laws.
  In particular, we also have
  \begin{equation*}
    \mathcal G(G, x, \mu) = \inf_{a\in A} \Big(L(x,a, \mu)\delta +\E\big[ G\big(\tilde{X}^{x,a}, m^{\mu,G}\big) \big]\Big)
\end{equation*}
with $\tilde{X}^{x,a} := x + b(x,\alpha, \mu)\delta + \tilde Z$ if $\tilde Z$ is a copy of $Z$ which is independent of $\cF_0\vee \sigma(Z)$.\ 
Moreover, 
\[
 \mathcal G(G, x, \mu)=  L\big(x,a^*(x,\mu,m^{\mu,G}), \mu\big)\delta +\E\Big[ G\Big(X^{x,a^*(x,\mu,m^{\mu,G})}, m^{\mu,G}\Big) \Big].
\]
In particular, for any random vector $\zeta\in \L^p\big(\Omega,\cF_0, \P)$, 
\begin{equation}\label{eq.conditionalexpectation.formula}
	\E\big[\mathcal G(G, \zeta, \mu)\big]= \E\Big[\essinf_{\alpha\in \mathcal A_S(\zeta)}\Big(L(\zeta,\alpha,\mu)\delta+\E\big[G\big(X^{\zeta,\alpha},m^{\mu,G}\big)\big| \zeta\big]\Big)\Big].
\end{equation}
\end{remark}

\begin{proposition}
\label{pro:asumption}
  Let Assumption \ref{Ass:multi} be satisfied.
  Then, the function $(x, \mu)\mapsto \cG(G,x,\mu)$ is
  \begin{enumerate}[label={(\roman*)}]
    \item bounded from below,
    \item bounded above by a constant times $(1+|x|^q)$ for every $x\in \R^d$ and $\mu\in \cP_p(\R^d)$,
    \item continuous on $\R^d\times \cP_p(\R^d,r)$ for every $r\geq 0$.
  \end{enumerate}
\end{proposition}

\begin{proof}
	$(i)$: This follows from the fact that the functions $G$ and $L$ are bounded from below.
	
	$(ii)$: Using the global assumptions on the growth of $G$, $L$, and $b$, there is a constant $C\geq 0$ such that 
	\begin{align*}
		\cG(G,x,\mu) &\le L(x, 0,\mu)\delta + \E\big[G(X^{x,0}, m^{\mu,G}) \big]\\
		&\le C\Big(1 + |x|^q + \E\big[\big|X^{x,0}\big|^q\big] \Big)\\
		&\le C\Big(1 + |x|^q + \E\big[|Z|^q\big]\Big).
	\end{align*}
	
	$(iii)$: 
    Let $(\mu^n)_{n\ge1}$ be a sequence converging to $\mu$, and let $\xi^n$ be a random variable with law $\mu^n$.
    By Lemma \ref{lem:optimization} and the proof of Theorem \ref{thm:exits.one.period}, the optimal control in the equilibrium $(m^{\mu^n,G}, \alpha^n)$ of the game with data $(\xi^n, Z, b,L,G)$ is given by $\alpha^n = a^*(\xi^n, \mu^n, m^{\mu^n,G})$.
    Since $\cP_p(\R^d,K)$ is compact, up to a subsequence, $m^{\mu^n,G}$ converges to a measure $\bar m$; and by continuity of $a^*$ (see Lemma \ref{lem:optimization}) and $b$, it follows that $X^{\xi^n, \alpha^n}$ converges to $X^{\xi,\bar\alpha}$ for some random variable $\xi$ with law $\mu$, and where $\bar \alpha = a^*(\xi, \mu, \bar m)$.
    Observe that $(\bar m, \alpha)$ is a mean field equilibrium for the problem with data $(\xi, Z,b,L,G)$.
    First, it is direct that the law of $X^{\xi,\bar \alpha}$ is $\bar m$.
    Optimality of $\bar \alpha$ follows by continuity of $L$ and $G$.
    In fact, for any admissible $\beta$, we have 
    \begin{align*}
        J_{\bar m}(\beta) & = \lim_{n\to\infty}\E\big[L(\xi^n, \beta, m^n) \delta + G(X^{\xi^n,\beta}, m^n) \big]\\
                          & \le \lim_{n\to\infty}\E\big[L(\xi^n, \alpha^n, m^n) \delta + G(X^{\xi^n,\alpha^n}, m^n) \big]\\
                          & = J_{\bar m}(\bar\alpha).
    \end{align*}
    Thus, by uniqueness of the mean field equilibrium, we have $\bar m = m^{\mu,G}$.
    Since every subsequence has a further subsequence that converges to $m^{\mu,G}$, continuity of $\mu\mapsto m^{\mu,G}$ follows.
    Now, the continuity of $(x,\mu)\mapsto \cG(G,x,\mu)$ is a direct consequence of the continuity assumptions in Assumption \ref{Ass:single.exists}
\end{proof}

We end this subsection by showing that the value function $\cG$ is also Lasry-Lions monotone, provided that both $L$ and $G$ are also Lasry-Lions monotone.
We introduce the following condition:
\begin{assumption}
\label{ass:LL.condition}
  The running cost $L$ is separated, that is, there are two functions $F\colon\R^d\times \cP_p(\R^d)\to \R$ and $L_0\colon \R^d\times A\to \R$ such that 
  $$L(x,a,m) = L_0(x,a) + F(x,m)$$ for all $(x,a,m)\in \R^d\times A\times \cP_p(\R^d)$ and $F$ satisfies the Lasry-Lions monotonicity condition.
\end{assumption}
\begin{proposition}
\label{eq:LL.propagates}
  Let Assumptions \ref{Ass:multi} and \ref{ass:LL.condition} be satisfied.\
  Then the map $(x,\mu)\mapsto \mathcal G(G,x,\mu)$ also satisfies the Lasry-Lions condition.
\end{proposition}
\begin{proof}
  Let $\mu^1,\mu^2\in \cP_p(\R^d)$ and $\xi^1, \xi^2$ be two $\cF$-measurable random variables (replacing the initial condition $\xi$), which are independent from $Z$ with laws $\mu^1, \mu^2$, respectively.
  Moreover, let $(m^1,\alpha^1)$ denote the MFE of the game with initial state $\xi^1$, and $(m^2,\alpha^2)$ the MFE of the game with initial state $\xi^2$.
  By Remark \ref{rem:dependence.G}, we have
  \begin{equation*}
    \begin{cases}
      \cG(G,\xi^1,\mu^1) = L(\xi^1, \alpha^1, \mu^1)\delta + \E\big[G(X^{\xi^1,\alpha^1}, m^1)\big| \xi^1\big],\\
        \cG(G,\xi^2,\mu^2) = L(\xi^2, \alpha^2, \mu^2)\delta +\E\big[G(X^{\xi^2,\alpha^2}, m^2)\big| \xi^2\big],
    \end{cases}
  \end{equation*}
  and 
  \begin{equation*}
    \begin{cases}
      \cG(G,\xi^1,\mu^2) \le L(\xi^1, \alpha^1, \mu^2)\delta +\E\big[ G(X^{\xi^1,\alpha^1}, m^2)\big|  \xi^1\big],\\
        \cG(G,\xi^2,\mu^1) \le L(\xi^2, \alpha^2, \mu^1)\delta +\E\big[ G(X^{\xi^2,\alpha^2}, m^1)\big| \xi^2\big].
    \end{cases}
  \end{equation*}
  Hence,
  \begin{align*}
    \int_{\R^d}\big( &\cG(G, x, \mu^1) - \cG(G, x, \mu^2) \big)(\mu^1 - \mu^2)(dx)\\
    & = \E\big[\cG(G,\xi^1, \mu^1) - \cG(G, \xi^1, \mu^2)\big] - \E\big[\cG(G, \xi^2, \mu^1) - \cG(G,\xi^2, \mu^2)\big]\\
    & \ge \E\big[L(\xi^1, \alpha^1, \mu^1)\delta + G(X^{\xi^1,\alpha^1}, m^1)\big] - \E\big[L(\xi^1, \alpha^1, \mu^2)\delta  + G(X^{\xi^1,\alpha^1}, m^2)\big]\\
    &\quad  - \E\big[L(\xi^2, \alpha^2, \mu^1)\delta  + G(X^{\xi^2,\alpha^2}, m^1)\big]  + \E\big[L(\xi^2, \alpha^2, \mu^2)\delta + G(X^{\xi^2,\alpha^2}, m^2)\big]\\
    & = \delta\E\big[ L(\xi^1,\alpha^1, \mu^1) - L(\xi^1, \alpha^1, \mu^2) - L(\xi^2,\alpha^2,\mu^1) + L(\xi^2,\alpha^2,\mu^2) \big] \\
    &\quad + \E\big[G(X^{\xi^1,\alpha^1}, m^1) - G(X^{\xi^1,\alpha^1},m^2) - G(X^{\xi^2,\alpha^2},m^1) + G(X^{\xi^2,\alpha^2}, m^2)]\\
    & = \delta \int_{\R^d}\big(F(x,\mu^1) - F(x, \mu^2)\big)(\mu^1 - \mu^2)(dx) + \int_{\R^d}\big(G(x,m^1) - G(x, m^2)\big)(m^1 - m^2)(dx)\\
    &\ge 0,
  \end{align*}
  where, in the second to last step, we used that $m^1$ and $m^2$ are the the laws of $X^{\xi^1,\alpha^1}$ and $X^{\xi^2,\alpha^2}$, respectively, and the last step follows from the fact that both $F$ and $G$ satisfy the Lasry-Lions monotonicity condition. The proof is complete.
\end{proof}

\subsection{Examples}
We give two simple examples of this 1-period game. In this section, we keep $\delta = 1$ and consider one-dimensional state and control spaces, i.e. $d=1$. We point out that our global assumptions are to be seen as sufficient conditions. In particular, the first example does not satisfy all our global assumptions but still falls into our setup, see also Remark \ref{rem.sufficient.conditions} below.
\subsubsection{Linear-quadratic games}
\label{Sec.LQ.example}
Consider the linear drift $b(x, a, m) = a$. Given a population distribution $m\in \cP_p(\R)$, the representative agent attempts to minimize the mean squared error between $X^\alpha$ and the population mean $\underline{m} \coloneqq \int_{\R}xm(dx)$ by exerting an effort $\alpha$ with a quadratic cost. An MFE $(\widehat{\alpha},\widehat{m})$ solves the following system for some $c > 0$ and $c_L > 0$:
\begin{equation*}
\begin{cases}
    \widehat{\alpha} \text{ minimizes } J_{\widehat{m}}(\alpha) \coloneqq \E\sqbra{\abs{X^{\alpha} - \underline{\widehat{m}}}^2 + c_L|\xi - \underline{m_\xi}|^2 +  c\alpha^2} \text{ over }\cA_S, \\
    \text{where }X^\alpha = \xi + \alpha + Z,\quad \P \circ (X^{\widehat{\alpha}})^{-1} = \widehat{m}.
\end{cases}
\end{equation*}
We now show that the following pair $(\widehat{\alpha},\widehat{m})$ is the unique MFE:$$\widehat{m} \coloneqq \P\circ \pa{\frac{c}{1+c}\xi + \frac{1}{1+c}\E[\xi] + Z}^{-1},\qquad \widehat{\alpha} \coloneqq \frac{\underline{\widehat{m}} - \xi}{1+c}.$$
For a given $m$, minimizing $J_m(\alpha)$ is equivalent to minimizing $$\E\sqbra{(\alpha + \xi)^2 - 2\underline{m}\alpha + c\alpha^2}.$$
The first order condition gives the unique minimizer $\alpha(m) = \frac{\underline{m} - \xi}{1 + c}$, and the corresponding state $$X^{\alpha(m)} = \frac {c}{1+c}\xi + \frac{1}{1+c}\underline{m} + Z.$$
The fixed point condition implies $\underline{m} = \E[X^{\alpha(m)}]$, which means that any equilibrium $m$ must satisfy $\underline{m} = \E[\xi]$, so in equilibrium we must have 
\begin{equation}\label{eq:Xalpha_LQ}
    X^{\alpha} = \frac{c}{1+c}\xi + \frac{1}{1+c}\E[\xi] + Z.
\end{equation}
Taking its law satisfies both optimality and the fixed point condition. 
In fact, this equilibrium is unique since the optimizer for each $m$ is unique and \eqref{eq:Xalpha_LQ} does not depend on $m$.

\subsubsection{Bounded drift}
Now consider $b(x, a, m) = kb(a)$, where $b$ is a sigmoid function and $k >0$ is the scale. An MFE $(\widehat{\alpha},\widehat{m})$ solves the following system for some $c > 0$:
\begin{equation*}
\begin{cases}
    \widehat{\alpha} \text{ minimizes } J_{\widehat{m}}(\alpha) \coloneqq \E\sqbra{\abs{X^{\alpha} - \underline{\widehat{m}}}^2 + c\alpha^2} \text{ over }\cA_S, \\
    \text{where }X^\alpha = \xi + kb(\alpha) + Z,\quad \P \circ (X^{\widehat{\alpha}})^{-1} = \widehat{m}.
\end{cases}
\end{equation*}
For a concrete example, take $b(a) = \tanh{a}$. Note that $|b|$, $|b'|$ and $||b''||$ are all bounded by $1$. In particular, an equilibrium must satisfy
$$|\underline{m} - \E[\xi]| \leq k.$$
It is therefore equivalent to consider the truncated cost
$$J^k_{\widehat{m}}(\alpha) \coloneqq \E\sqbra{\abs{X^{\alpha} - \underline{\widehat{m}}_k}^2 + c\alpha^2}, $$
where $\underline{m}_k$ is the projection of $\underline{m}$ onto $[\E[\xi] - k, \E[\xi] + k]$. The first order condition states that $$\E\sqbra{2(\xi + b(a))b'(a) - 2b'(a)\underline{m}_k + 2ca/k} = 0.$$
Since $b$ is a sigmoid function, $\frac{d}{da}J^k_{m}(a)$ goes to $-\infty$ as $a \to -\infty$ and goes to $\infty$ as $a \to \infty$. Taking again derivatives, we get $$\frac{d^2}{da^2}J^k_{m}(a) = 2\E\sqbra{c/k + |b'(a)|^2 + b''(a)(\underline{m}_k - \xi - b(a))} \geq 2c/k - 2|\underline{m}_k - \E[\xi]| - 2.$$
Therefore, if $c > k^2 + k$, then $J^k_{\widehat{m}}$ is guaranteed to have a unique minimizer.

\section{Multi-period MFGs}
\label{sec:multip}

In the multi-period case, the game is set over discrete-time s $\{t_0,t_1,\dots, t_k\}$ with $\delta:=t_{i+1} - t_i>0$.
In this setting, we further assume that the probability space $(\Omega, \cF,\P)$ carries a $d$-dimensional process $Z=(Z_{t_i})_{i=0,\dots,k}$ with $Z_{t_0} = 0$ and $Z_{t_{i+1}}-Z_{t_i}\sim\nu_{i}\in \cP_p(\R^d)$ independent of the $\sigma$-algebras $\cF_0$ and $\sigma(Z_{t_0},\ldots, Z_{t_i})$, for $i=1,\ldots, k-1$. Moreover, we assume that $\E(Z_{t_i})=0$ for $i=1,\ldots, k$, so that $Z$ is a martingale w.r.t.\ the filtration $\F:=(\cF_{t_i})_{i=1,\ldots, k-1}$.\

Throughout, let $\xi\in \L^p(\Omega,\cF_0,\P)$. For a flow of measures $(m_{t_i})_{i=0,\dots,k}\subset \cP_p(\R^d)$ with $m_{t_0} = m_\xi$ and an  $\cF$-adapted process $\alpha$ taking values in $A$, we consider the associated state process
\begin{equation}
	\label{eq:M.per.state}
	X^\alpha_{t_{i}} = \xi + \sum_{j=0}^{i-1}b(X^\alpha_{t_j},\alpha_{t_j}, m_{t_j})\delta +  Z_{t_{i}},\quad\text{with}\quad i=1,\dots,k\quad \text{and}\quad X^\alpha_{t_0} =\xi.\
\end{equation}
It is immediate that $X^\alpha$ can be defined recursively through $X_{t_0}^\alpha = \xi$ and
\begin{align*}
	X_{t_i}^\alpha & = X^{\alpha}_{t_{i-1}} + b(X^\alpha_{t_{i-1}}, \varphi(t_{i-1},X^\alpha_{t_{i-1}}), m_{t_{i-1}})\delta + (Z_{t_i} - Z_{t_{i-1}})\\
	&=X^{\alpha}_{t_{i-1}} + b(X^\alpha_{t_{i-1}}, \alpha_{t_{i-1}}, m_{t_{i-1}})\delta + (Z_{t_i} - Z_{t_{i-1}})\quad\text{for }i=1,\ldots,k.
\end{align*}	
We denote by $\cA_M$ the set of admissible controls in the multi-period game, which are assumed to be of Markovian type. These are stochastic processes $(\alpha_{t_i})_{i=0,\dots,k}$ such that, for all $i=1,\ldots,k-1$, $\alpha_{t_i} = \varphi(t_i, X_{t_{i}}^\alpha)$ for some measurable function $\varphi$ such that for each $i$, the random variable $\alpha_{t_i}$ is in $\L^p$, where $(X^\alpha_{t_i})_{i=0,\dots,k}$ is a the associated state process given in \eqref{eq:M.per.state}.\ The representative agent seeks to minimize the total cost
\begin{equation}
\label{eq:const.mult.period}
  J_m(\alpha) :=\E\bigg[\sum_{i=0}^{k-1}L(X^{\alpha}_{t_i},\alpha_{t_i}, m_{t_i})\delta + G(X^{\alpha}_{t_k}, m_{t_k}) \bigg]
\end{equation}
among all Markovian controls $\alpha\in \cA_M$.

An MFE is a pair $(\alpha,m)$ such that $m=(m_{t_i})_{i=0,\dots,k}\in \cP_p(\R^d)^{k+1}$, $\alpha\in \cA_M$ is a minimizer of $J_m$ and 
\begin{equation*}
  \P\circ (X^\alpha_{t_i})^{-1} = m_{t_i}\quad \text{for all }i=0,\dots,k.
 \end{equation*} 
  We will refer to this problem as \emph{the multi-period mean field game}.
 This problem is fully determined by the data $(\xi,Z, b, L, G,k)$.

\subsection{Existence, uniqueness and pasting}
\label{sec:exists.mult.periods}
In this section, we prove existence and uniqueness of an MFE for the multi-period mean field game.
We will show, in addition, that the MFE is obtained as a concatenation of suitable single-period MFGs.
Recall the function $\cG$ defined in Equation \ref{eq:def.calG}.
In order to introduce these single-period MFGs, we first iteratively define the functions $g_i$ as 
\begin{equation*}
\begin{cases}
  g_k(x,m) := G(x,m),\\
  g_{i-1}(x,m) = \mathcal G(g_i,x,m),\quad i=k,\dots,1.
\end{cases}
\end{equation*}
The main result of this section is the existence, uniqueness and characterization of equilibria in the multi-period game under the following assumption.

\begin{assumption}\label{Ass:multi.exists}
Assume that Assumption \ref{Ass:multi} and Assumption \ref{ass:LL.condition} are satisfied.\ Moreover, for $i=0,\ldots, k-1$, the optimization problem
	\[
	\inf_{a\in A}\bigg( L(x,a,\mu)\delta +\int_{\R^d} g_i\big(x+b(x,a,\mu)\delta +z, m\big)m_{Z^i}(dz)\bigg)
	\]
	has a unique solution for all $x\in \R^d$, $\mu\in \cP_p(\R^d)$, and $m\in \cP_p(\R^d)$.
\end{assumption}	
	
	Furthermore, iteratively define the sets $\cA_S^i$ and the pairs $(\alpha^{i},m^{i})\in \cA_S^i\times \cP_1(\R^d)$ as follows:
	\begin{itemize}
		\item Let $m^0= m_\xi$ be the initial state of the multi-period MFG and, by convention, 
		$X^{\alpha^{0}}:=\xi$.
		\item For each $i = 1,\dots,k$, let $\cA_S^i = \cA_S(X^{\alpha^{i-1}})$ be the set of $p$-integrable random variables that are $\cF^{X^{\alpha^{i-1}}}$-measurable, where
		\begin{equation*}
			X^{\alpha^{i}} := X^{\alpha^{i-1}} + b(X^{\alpha^{i-1}},\alpha^{i}, m^{i-1})\delta + (Z_{t_{i}} - Z_{t_{i-1}}),
		\end{equation*}
		and let $(\alpha^{i}, m^{i})$ be the MFE of the single-period MFG with data $(X^{\alpha^{i-1}},Z_{t_i}-Z_{t_{i-1}}, b , L,g_i)$.
		In particular, we have $m^{i} = \P\circ (X^{\alpha^i})^{-1}$ and $\alpha^{i}$ is optimal for
		\begin{equation*}
			J_{m^i}(\beta) = \E\big[L(X^{\alpha^{i-1}},\beta,m^{i-1})\delta+ g_i(X^{\beta}, m^i) \big]
		\end{equation*}
		with $X^{\beta} := X^{\alpha^{i-1}} + b(X^{\alpha^{i-1}},\beta, m^{i-1})\delta + (Z_{t_{i}} - Z_{t_{i-1}})$ and $\beta \in \cA^i_S$. 
	\end{itemize}

\begin{theorem}
\label{thm.rep.full}
  Let Assumption \ref{Ass:multi.exists} is satisfied.\ Then, for each $i=1,\dots, k$, the MFE $(\alpha^i,m^{i})\in \cA_S^i\times \cP_p(\R^d)$ of the single-period MFG with data $(X^{\alpha^{i-1}},Z_{t_i}-Z_{t_{i-1}}, b , L,g_i)$ exists and is unique.
  Moreover, with $m_{t_0} = m_{\xi}$, the pair of processes\footnote{By convention we set $\alpha_{t_k} = 0$.} $(\alpha_{t_i},m_{t_i})_{i=0,\dots,k}$ given by
  \begin{equation}
  \label{eq:def.MFE.thm}
    \alpha_{t_{i-1}} = \alpha^i\quad \text{and}\quad m_{t_i} = m^i\quad \text{for}\quad i=1,\dots,k
  \end{equation}
    is an MFE of the multi-period MFG with data $(\xi,Z, b, L, G,k)$.
  If $b$ does not depend on the measure argument, then this MFE is unique.
\end{theorem}
\begin{proof}
  By Proposition \ref{eq:LL.propagates}, all the functions $g_i$ satisfy the Lasry-Lions monotonicity condition and, by Proposition \ref{pro:asumption}, the functions $g_i$ satisfy the assumptions on $G$ in Assumption \ref{Ass:single.exists}.
  Therefore, it follows from Theorem \ref{thm:exits.one.period} and Proposition \ref{prop:unique.single.period} that the family $(\alpha^i,m^i)$ exists and is uniquely defined for each $i=0,\dots,k$.
  We split the rest of the proof into three steps.

\medskip

  \emph{Step1: Admissibility of flow and consistency.}
  In this step, we show that the process $(\alpha_{t_i})_{i=0,\dots,k}$ forms an admissible control and that the flow of measures $(m_{t_i})_{t=0,\dots,k}$ satisfies the consistency condition.
  By construction, the process $(X_{t_i})_{i=1,\dots,k}$ defined as
  \begin{equation*}
    X_{t_i} := X^{\alpha^i}
  \end{equation*} 
  satisfies $X_{t_0} = \xi$ and for every $i=1,\dots,k$,
  \begin{align*}
   X_{t_{i}} = X^{\alpha^i} &= X^{\alpha^{i-1}} + b(X^{\alpha^{i-1}},\alpha^{i}, m^{i-1})\delta + (Z_{t_{i}} - Z_{t_{i-1}}) \\
  & = X_{t_{i-1}} + b(X_{t_{i-1}},\alpha_{t_{i-1}}, m_{{t_{i-1}}})\delta + (Z_{t_{i}} - Z_{t_{i-1}}),
\end{align*}
where we used that $m_{t_{i-1}} = m^{i-1}$.
That is, $X$ satisfies the state equation
\begin{equation*}
  X_{t_{i}} = \xi + \sum_{j=0}^{i-1}b(X_{t_{j}},\alpha_{t_{j}}, m_{{t_{j}}})\delta + Z_{t_{i}}.
\end{equation*}
Moreover, the process $\alpha$ is admissible since each $\alpha_{t_i}=\alpha^{i+1} \in \cA_S^{i+1}$ is by definition $p$-integrable and $\cF^{X^{\alpha^i}}=\cF^{X_{t_{i}}} $-measurable.
The consistency condition follows by the fact that $(m^i,\alpha^i)$ is the MFE of the single-period game:
  $m_{t_{i}} = m^{i} = \P\circ (X^{\alpha^{i}})^{-1} = \P\circ(X_{t_{i}})^{-1}$ for all $i=0,\dots,k$.

\medskip

\emph{Step 2: Optimality of the control.}
In this step, we show that the control process $(\alpha_{t_i})_{i=0,\dots,k}$ is optimal.\
This follows by (one side of) the dynamic programming principle.\
For any process $\beta = (\beta_{t_i})_{i = 0, \dots, k} \in \cA_M$ and any random variable $\zeta$ and $i\le j\le k$, we use the notation

\begin{equation*}
    X^{t_i,\zeta,\beta}_{t_{j}} = \zeta + \sum_{\ell=i}^{j-1}b(X^{t_i, \zeta, \beta}_{t_\ell},\beta_{t_l}, m_{t_\ell})\delta +  Z_{t_{\ell}},\quad\text{with} 
    \quad X^{t_i,\zeta,\beta}_{t_i} =\zeta.
\end{equation*}
Note that $X^{t_i, \zeta, \beta}_{t_j}$ only depends on $(\beta_i, \dots, \beta_{j-1})$. In particular, if $j = i + 1$, under a slight abuse of notation, we may take $\beta \in \cA_{S}(\zeta)$ instead of $\beta \in \cA_{M}$ to define $X^{t_i, \zeta, \beta}_{t_{i+1}}$. 

For the (deterministic) measure flow $m$ defined in \eqref{eq:def.MFE.thm},
we have
\begin{align*}
  \E\bigg[\sum_{i=0}^{k-1}L(X_{t_i},\alpha_{t_i},& m_{t_i})\delta + G(X_{t_k}, m_{t_k}) \bigg] \ge \inf_{\beta \in \cA_M} \E\bigg[\sum_{i=0}^{k-1}L(X_{t_i}^{t_0,\xi,\beta},\beta_{t_i}, m_{t_i})\delta + g_k(X_{t_k}^{t_0,\xi,\beta}, m_{t_k}) \bigg]\\
  & \ge \inf_{\beta \in \cA_M} \E\Bigg[\sum_{i=0}^{k-2}L(X_{t_i}^{t_0,\xi,\beta},\beta_{t_i}, m_{t_i})\delta \\
  &\qquad  + \inf_{\tilde{\beta} \in \cA_{S}(X^{t_0, \xi, \beta}_{t_{k-1}})} \E\bigg[L(X_{t_{k-1}}^{t_0,\xi,\beta},\tilde{\beta}, m_{t_{k-1}})\delta  + g_k(X_{t_k}^{t_{k-1},X_{t_{k-1}}^{t_0,\xi,\beta},\tilde{\beta}}, m_{t_k}) \bigg]\Bigg]\\
    & \ge \inf_{\beta \in \cA_M} \E\Bigg[\sum_{i=0}^{k-2}L(X_{t_i}^{t_0,\xi,\beta},\beta_{t_i}, m_{t_i})\delta \\
    &\quad +\essinf_{\tilde{\beta} \in \cA_S(X_{t_{k-1}}^{t_0,\xi,\beta})} \E\bigg[L(\zeta,\tilde{\beta}, m_{t_{k-1}})\delta  + g_k(X_{t_k}^{t_{k-1},\zeta,\tilde{\beta}}, m_{t_k})\bigg| \zeta=X_{t_{k-1}}^{t_0,\xi,\beta}  \bigg]\Bigg]\\
    & = \inf_{\beta \in \cA_M} \E\Bigg[\sum_{i=0}^{k-2}L(X_{t_i}^{t_0,\xi,\beta},\beta_{t_i}, m_{t_i})\delta  \\
    &\quad+ \essinf_{\tilde{\beta} \in \cA_S(X_{t_{k-1}}^{t_0,\xi,\beta})} \E\bigg[L(\zeta,\tilde{\beta}, m^{{k-1}})\delta  + g_k (X_{t_k}^{t_{k-1},\zeta,\tilde{\beta}}, m^{k})\bigg| \zeta=X_{t_{k-1}}^{t_0,\xi,\beta}  \bigg]\Bigg]\\
    & = \inf_{\beta \in \cA_M} \E\bigg[\sum_{i=0}^{k-2}L(X_{t_i}^{t_0,\xi,\beta},\beta_{t_i}, m_{t_i})\delta  + g_{k-1}(X_{t_{k-1}}^{t_0,\xi,\beta}, m^{{k-1}})\bigg],
\end{align*} 
where  the second inequality uses definition of admissibility of controls and $g_k = G$, the third inequality uses that $X^{t_0,\xi, \beta}_{t_k} = X_{t_k}^{t_{k-1}, X_{t_{k-1}}^{t_0,\xi, \beta},\beta_{t_k - 1}}$, the fourth inequality uses that $m_{t_i} = m^i$, and the last equality follows from Remark \ref{rem:dependence.G}.

Applying the same idea, iteratively for $i=k-1,\dots,1$, we arrive at
\begin{align*}
  \inf_{\beta \in \cA_M} \E\bigg[\sum_{i=0}^{k-2}L(X_{t_i}^{t_0,\xi,\beta},\beta_{t_i}, m_{t_i})\delta  + g_{k-1}(X_{t_{k-1}}^{t_0,\xi,\beta}, m^{{k-1}})\bigg]
  \ge \inf_{\beta \in \cA_S}\E\Big[L(\xi,\beta, \cL(\xi))\delta + g_1(X^{t_0,\xi,\beta}_{t_1}, m^{1}) \Big].
\end{align*}
Next, using the fact that $(\alpha^1,m^1)$ is the MFE for the single-period MFG with data $(\xi, Z_{t_1},b,L,g_1)$, it follows from definition of $g_1$ that
\begin{align}
\notag
  &\inf_{\beta \in \cA_S}\E\Big[L(\xi,\beta, \cL(\xi))\delta + g_1(X^{t_0,\xi,\beta}_{t_1}, m^{1}) \Big]\\\notag
  & \quad = \E\Big[L(\xi,\alpha^1, \cL(\xi))\delta + g_1(X^{t_0,\xi,\alpha^1}_{t_1}, m^{1}) \Big]\\\label{eq:step.before.directed}
  &\quad = \E\Big[L(\xi,\alpha_{t_0}, \cL(\xi))\delta + \essinf_{\tilde\beta\in \cA_S(X^{t_0,\xi,\alpha^1}_{t_1})}\E\Big[L(\zeta, \tilde\beta,m^1)\delta + g_{2}(X_{t_2}^{t_1,\zeta,\tilde\beta}, m^{2}) \Big| \zeta  = X^{t_0,\xi,\alpha^1}_{t_1}\Big]\Big].
\end{align}
One readily verifies that the set $\Big\{\E\Big[L(\zeta, \tilde \beta,m^1)\delta + g_{2}(X_{t_2}^{t_1,\zeta,\tilde\beta}, m^{2}) \Big| \zeta  = X^{t_0,\xi,\alpha^1}_{t_1}\Big]:  \tilde\beta\in \cA_S(X^{t_0,\xi,\alpha^1}_{t_1}) \Big\}$ is directed downwards in the sense of \cite[Theorem A.37]{FS3dr}, which implies that there is a sequence $(\beta^n)_{n\geq 1} \in \cA_S(X_{t_1}^{t_0}, \xi, \alpha^1)$ such that 
\begin{align*}
  &\essinf_{\tilde\beta\in \cA_S(X^{t_0,\xi,\alpha^1}_{t_1})}\E\Big[L(\zeta, \tilde\beta,m^1)\delta + g_{2}(X_{t_2}^{t_1,\zeta,\tilde\beta}, m^{2}) \Big| \zeta  = X^{t_0,\xi,\alpha^1}_{t_1}\Big]\\
  &\quad = \lim_{n\to\infty}\E\Big[L(\zeta, \beta^n,m^1)\delta + g_{2}(X_{t_2}^{t_1,\zeta,\beta^n}, m^{2}) \Big| \zeta  = X^{t_0,\xi,\alpha^1}_{t_1}\Big]
\end{align*}
as a monotone sequence.
Thus, by monotone convergence theorem, we have
\begin{align*}
  &\E\Bigg[\essinf_{\tilde\beta\in \cA_S(X^{t_0,\xi,\alpha^1}_{t_1})}\E\Big[L(\zeta, \tilde\beta,m^1)\delta + g_{2}(X_{t_2}^{t_1,\zeta,\tilde\beta}, m^{2}) \Big| \zeta  = X^{t_0,\xi,\alpha^1}_{t_1}\Big]\Bigg]\\
  &\quad = \lim_{n\to\infty}\E\bigg[L(X^{t_0,\xi,\alpha^1}_{t_1}, \beta^n,m^1)\delta + g_{2}(X_{t_2}^{t_1,X^{t_0,\xi,\alpha^1}_{t_1},\beta^n}, m^{2})\bigg]\\
  &\quad \ge \inf_{\tilde\beta\in \cA_S(X^{t_0,\xi,\alpha^1}_{t_1}) } \E\bigg[L(X^{t_0,\xi,\alpha^1}_{t_1}, \tilde\beta,m^1)\delta + g_{2}(X_{t_2}^{t_1,X^{t_0,\xi,\alpha^1}_{t_1},\tilde\beta}, m^{2})\bigg].
\end{align*}
Thus, coming back to \eqref{eq:step.before.directed} and using the fact that $(\alpha^i,m^i)$ is the MFE for the single-period MFG with data $(X_{t_{i-1}}, Z_{t_i}-Z_{t_{i-1}}, b, L,g_i)$, we have that
\begin{align*}
    \inf_{\beta \in \cA_S}\E\Big[L(\xi,\beta, \cL(\xi))\delta &+ g_1(X^{t_0,\xi,\beta}_{t_1}, m^{1}) \Big]\\
    & \ge \E\Big[L(\xi,\alpha_{t_0}, \cL(\xi))\delta + \E\Big[L(X^{t_0,\xi,\alpha^1}_{t_1}, \alpha^2,m^1)\delta + g_{2}(X_{t_2}^{t_0,\xi,\alpha^2}, m^{2}) \Big]\Big]\\ 
  &= \E\bigg[\sum_{i=0}^1L(X^{t_0,\xi,\alpha}_{t_i}, \alpha_{t_i},m_{t_i})\delta + g_{2}(X_{t_2}^{t_0,\xi,\alpha^2}, m^{2}) \bigg].
\end{align*}
Using exactly the same idea subsequently for $i=2,\dots,k-1$, we arrive at
\begin{align*}
\E\bigg[\sum_{i=0}^1L(X^{t_0,\xi,\alpha^i}_{t_i}, \alpha_{t_i},m_{t_i})\delta + g_{2}(X_{t_2}^{t_0,\xi,\alpha^2}, m^{2}) \bigg]   &\ge \E\bigg[\sum_{i=0}^{k-1}L(X_{t_i}, \alpha_{t_i},m_{t_i})\delta + g_k(X_{t_k}^{t_0,\xi,\alpha^k}, m_{t_k}) \bigg]\\
   &= \E\bigg[\sum_{i=0}^{k-1}L(X_{t_i}, \alpha_{t_i},m_{t_i})\delta + G(X_{t_k}, m_{t_k}) \bigg],
\end{align*}
which proves optimality, and therefore, $(\alpha,m)$ is an MFE.

  \emph{Step 3: Uniqueness.}
  When $b$ does not depend on the measure argument, using Assumption \ref{ass:LL.condition}, the proof of uniqueness is essentially the same as in the single-period case, which is performed in Proposition \ref{prop:unique.single.period}.
  It is thus omitted in order to avoid repetition.
\end{proof}
\begin{remark}\label{rem.sufficient.conditions}
  Although Theorem \ref{thm.rep.full} requires the coefficients to satisfy the Lasry-Lions monotonicity condition (see Assumption \ref{ass:LL.condition}), a close observation of the proof reveals that what is actually needed is that each single period game between $t_i$ and $ t_{i+1}$ has a unique MFE. 
 In particular, uniqueness in each period ensures that the terminal cost functions $\{g_i\}_{i \leq k}$ are well-defined. In Section \ref{section:Two-periodLQ} below, we provide an explicit example for the pasting procedure where the single-period game has a unique equilibrium without satisfying Assumption \ref{ass:LL.condition}.

More generally, this pasting procedure poses the larger question of ``selection of equilibria'':
  Without uniqueness, one may need to appropriately select which equilibria should be pasted together.
  But this is not straightforward since, in this case, the new terminal cost functions we introduced may, a priori, not even be well-defined.
  However, we mention that pasting results for equilibria have been obtained, without uniqueness, in \cite{mou.zhang.2024minimal}. 
  In this work, the authors paste over time the minimal equilibrium in a particular class of extended continuous-time MFGs.
  Their approach is based on the master equation.
  These results suggest that pasting might be possible for minimal and maximal equilibria of discrete-time submodular MFGs, see \cite{dianetti2022strong, dianetti.ferrari.fischer.nendel.2019, dianetti.ferrari.fischer.nendel.2022unifying}. 
  We leave this question for future research.
\end{remark}

\subsection{A two-period linear-quadratic game}\label{section:Two-periodLQ}
We revisit the example given in Section \ref{Sec.LQ.example} and consider its two-period extension still with $\delta = 1$. 
In the notation of the multi-period setup, $$L(x, a, m) = ca^2 + c_L(x - \underline{m})^2\quad\text{and} \quad g_2(x, m) = G(x, m) = (x - \underline{m})^2.$$ 
Given a random variable $X_1$ and the increment $\Delta Z_i \coloneqq Z_i - Z_{i-1}$, the equilibrium $(\alpha^2, m^2)$ for the second period is given as a function of $X_1$ and its law by:  
$$m^2 = \mathcal{L}\Big(\frac{c}{1+c} X_1 + \frac{1}{1+c}\E[X_1] + \Delta Z_2\Big)\quad\text{and} \quad \alpha^2 =\psi_2(X_1; m^2) \equiv \frac{\underline{m^2} - X_1}{1 + c},$$
where $\underline{m^2} = \int x m^2(dx) = \E[X_1]$. Note that $m^2$ here depends (only) on the law of $X_1$, which will be the dummy variable $m$ in the calculations below. In particular, $\underline{m^{2,m}} = \underline{m}$ for any $m \in \mathcal{P}_1(\R)$. Then, using the equilibrium $(\alpha^2, m^2)$ for the second period, from \eqref{eq:def.calG} we get
\begin{align*}
    g_1(x, m) &= \mathcal{G}(G, x, m) = \inf_{a \in A}\E\Big[L(x, a, m) + G\big(x + a + \Delta Z_2, m^{2,m}\big)\Big]\\
    & = \E\Big[L(x, \psi_2(x; m^{2,m}), m) + G\big(x + \psi_2(x; m^{2,m} + \Delta Z_2, m^{2,m}\big)\Big]\\
    & = c_L(x - \underline{m})^2 + \frac{c}{(1+c)^2}(x - \underline{m})^2 + \E\Big[\big(\frac{c}{1+c}(x - \underline{m}) + \Delta Z_2\big)^2\Big]\\
    & = \Big(c_L + \frac{c}{1+c}\Big)(x - \underline{m})^2 + \sigma^2.
\end{align*}
The function $g_1$ is taken as terminal cost of the first period game:
For fixed $m^1$, we aim to minimize $$J_{m^1}(\alpha^1) = \E\Big[L(\xi, \alpha^1, m_\xi) + g_1(X^{\alpha^1}, m^1)\Big],$$
which is equivalent to minimizing
$$\E\Big[c (\alpha^1)^2 + \tilde c\big((\alpha^1 + \xi)^2 - 2\underline{m^1}\alpha^1\big)\Big]$$
with $\tilde c \coloneqq c_L + \frac{c}{1+c}$. First order condition gives $$\alpha^1 = \frac{\tilde c}{c + \tilde c} (\underline{m^1} - \xi), \quad X^{\alpha^1} = \frac{c}{c + \tilde c}\xi + \frac{\tilde c}{c + \tilde c}\underline{m^1} + \Delta Z_1.$$
It is easy to verify that $m^1 = \mathcal{L}\Big(\frac{c}{c+\tilde c}\xi + \frac{\tilde c}{c + \tilde c}\E[\xi] + \Delta Z_1\Big)$ is an equilibrium for the first period. In particular, $\underline{m^1} = \E[\xi]$. Using the idea of pasting, $(m_{t_i}, \alpha_{t_i})_{i = 0, 1, 2}$ is the equilibrium of the two-period game, where
$$\begin{cases}
    m_{t_0} = m_\xi,\\
    m_{t_1} = m^1,\\
    m_{t_2} = m^2(m^1),
\end{cases}\;\text{and} \quad 
\begin{cases}
    \alpha_{t_0} = \alpha^1,\\
    \alpha_{t_1} = \psi_2(X^{\alpha^1}; m_{t_2}),\\
    \alpha_{t_2} = 0.
\end{cases}$$

\section{Donsker-type result for MFGs: Qualitative convergence}
\label{sec:convergence to continuous}

The goal of this section is to show, by using weak convergence of probability measures, that the multi-period MFG allows to approximate the game in continuous-time.
This section can be read independently from the previous ones as we will assume to be given discrete-time equilibria (without relying on the previous construction) and study their convergence as the step size $\delta$ tends to zero.

\textbf{Notation.} In the remaining sections, the set of admissible Markovian controls and the cost functional for the $k$-period MFG are denoted by $\cA_M^k$ and $J^k$, respectively.
\subsection{Continuous-time mean field game}
In the continuous-time case, we fix a time-horizon $T>0$ and assume that the complete probability space $(\Omega, \cF,\P)$ carries a $d$-dimensional Brownian motion $W$ and an independent initial state $\xi$ taking values in $\R^d$.
Denote by $\mathbb F = (\cF_t)_{t \in [0,T]}$ the right-continuous extension of the filtration generated by $\xi$ and $W$.
Given a progressive measurable integrable process $\alpha$
and a flow of measures $(m_t)_{t\in [0,T]} $, where $m_t\in \cP_2(\R^d)$ and $t\mapsto m_t$ is continuous, denote by $X^\alpha$ the solution of
\begin{equation}
\label{eq:SDE.cont}
      X^\alpha_{t} = \xi + \int_0^tb(X^\alpha_{s},\alpha_{s}, m_{s})ds + \sigma W_t,
\end{equation}
where $\sigma$ is an invertible $d\times d$ matrix with Frobenius norm $|\sigma|$ and such that $\sigma\sigma^\top>0 $ (in the sense of positive definite matrices). 
A process $\alpha $ is said to be an admissible (continuous-time) Markovian control if $\alpha_t =  \alpha(t, X_t^\alpha)$ for some measurable function $\alpha$ and such that \eqref{eq:SDE.cont} admits a unique strong solution.
With slight abuse of notation, we will often use the same notation $\alpha$ for the control $\alpha(t,X^\alpha_t)$ and the map $\alpha$. 
The set of admissible (continuous-time) Markovian controls is denoted by $\mathcal A _M$. 

Given a (continuous-time) measure flow $m:[0,T] \to \cP_2 (\R^d)$,  the representative agent seeks to minimize the total cost
\begin{equation*}
  J_m(\alpha) :=\E\bigg[\int_0^TL(X^\alpha_s,\alpha_{s}, m_{s})ds + G(X^\alpha_{T}, m_{T}) \bigg]
\end{equation*}
over $ \alpha \in \cA _M$. 
A control $\alpha \in \cA _M$ is said to be optimal (for $m$) if $J_m (\alpha) \leq J_m(\beta)$ for any $\beta \in \cA _M$.

We say that a pair $(\alpha, m)$ is an MFE for the continuous-time MFG if $\alpha$ is an admissible Markovian control (i.e., $\alpha \in \mathcal A _M $) which is optimal for $J_m$ and $\P\circ(X^\alpha_t)^{-1} = m_t$ for all $t\in [0,T]$.

In the remainder of the paper, we will refer to this problem as \emph{the continuous-time mean field game}.
The continuous-time MFG is fully determined by the data $(\xi,W, b, L, G,T)$ and it is rather well-studied.
In particular, general conditions (comparable to Assumption \ref{Ass:single.exists}) are known to guarantee the existence of an MFE, while its
uniqueness follows by the same requirements as those we presented in the multi-period case, see for instance \cite{cardaliaguet2019master,cardelbook1,cardelbook2,CarmonaLacker15} and the references therein.

\subsection{Time discretization}
In order to derive convergence of the discrete-time equilibria to the continuous-time equilibria, let us specialize the setting of Section \ref{sec:multip} as follows:
For $k \in \N$, we divide the time interval $[0,T]$ into $k$ periods, and we set 
\begin{equation}
\label{eq:Gaussian.case}
\delta := \frac Tk,\quad \text{and}\quad t_i := \delta i
\quad \text{for all}\quad i=0,\dots,k.
\end{equation}
We further assume that $Z^k =( Z^k_{t_i})_{i=0,...,k}$ is a sequence square-integrable martingales starting from zero, with independent increments, $\E[|Z^k_{t_{i+1}}-Z^k_{t_i}|^2] =  \delta$ and satisfying 
\begin{equation}\tag{Z1}\label{Z1}
 \lim_{k\to\infty}\E\Big[ \sup_{0\le t\le T}|Z^k_{t} -  W_t|^2\Big] = 0,
\end{equation}
where we identify $Z^k$ with its continuous-time interpolation as a constant process on the time interval $[t_i, t_{i+1})$.
In the notation of Section \ref{sec:multip}, we are considering here the case $Z = \sigma Z^k$.
\begin{example}[{Discretized Brownian motion}] The most natural example is to simply put
  \begin{equation*}
    Z^k_{t_i} := W_{t_i}, 
  \end{equation*}
  where $(W_t)_{t\in[0,T]}$ is the Brownian motion introduced above in the continuous-time model.
  By definition of the Brownian motion, it follows that for each $k$, $Z^k$ has independent increments with normal distribution $\cN(0, \delta I_d)$ where $I_d$ is the identity matrix of $\R^d$.
  Moreover, $Z^k$ is a square-integrable martingale and \eqref{Z1} is satisfied.
  In fact, it holds $\E[\sup_{0\le t\le T}|Z^k_t -  W_t|^2]\le C\delta$ for some constant $C>0$.
\end{example} 

For later reference, we let $J^k_m$ and $\cA^k_M$ be the cost function and the set of Markovian controls in the multi-period MFG with $k$ periods, defined in the previous section.
Given an admissible multi-period control $(\alpha_{t_i})_{i=0,\dots,k}\in \cA_M^k$, and a discrete measure flow $(m_{t_i})_{i=0,\dots,k}$ with $m_{t_i}\in \cP_{p}(\R^d)$, recall that the associated discrete-time state satisfies
\begin{align*}
  X_{t_i}
  &=X_{t_{i-1}} + b(X_{t_{i-1}}, \alpha_{t_{i-1}}, m_{t_{i-1}})\delta + \sigma (Z^k_{t_i} - Z^k_{t_{i-1}}).
\end{align*}
Let us introduce the continuous-time interpolations $\widehat \alpha = (\widehat \alpha_t)_{t\in[0,T]}$ and $\widehat m = (\widehat m_t)_{t\in[0,T]}$ of $\alpha$ and $m$, respectively, given by
\begin{equation}
\label{eq:cont.interp.alpha.m}
  \widehat \alpha_t := \sum_{i=0}^{k-1}\alpha_{t_i} (X_{t_i})\mathbf{1}_{[t_i, t_{i+1})}(t) \quad \text{and}\quad \widehat m_t := \sum_{i=0}^{k-1}m_{t_i}\mathbf{1}_{[t_i, t_{i+1})}(t), \ 
  \widehat m_{T} := m_{T}, \quad t \in [0,T].
\end{equation}
Of course, neither $\widehat\alpha$ nor $\widehat m$ are continuous processes.
They are simply defined on the continuous-time interval $[0,T]$.
We further define a continuous-time interpolation $\widehat X$ of $(X_{t_i})_{i=0,\dots,k}$ as the process satisfying
\begin{equation}
\label{eq:cont.time.approx}
  \widehat X_t = \xi + \int_0^{t}b(\widehat X_{\eta^k(s)}, \widehat \alpha_s, \widehat m_s)ds +  \sigma Z^k_{t}, \quad \text{with $\eta^k(s)=t_i$ for $t_i<s\le t_{i+1}$.}
\end{equation}
One easily checks that $\widehat X_{t_i} =X_{t_i} $ for all $i=0,\dots,k$ and
\begin{equation*}
  \widehat X_t = X_{t_i} + b(X_{t_i}, \alpha_{t_i}, m_{t_i})(t - t_i) + \sigma (Z^k_{t} - Z^k_{t_i})\quad  \text{for all}\quad t_i<t\le t_{i+1}.
 \end{equation*}
In particular, $\widehat X$ is a càdlàg (i.e., right-continuous  with left-limits) process.

\subsection{Stability of equilibria: from discrete-time to continuous-time}

Our first stability result will be discussed under the following requirements.

\begin{assumption}\
\label{Ass:Compact.result}
\begin{enumerate}[label={(\roman*)}]  
  \item \eqref{Z1} holds;
  \item $\xi$ is such that $\E[|\xi|^4] < \infty$;
  \item $b(a,x,m) = a + b_0(x,m)$ and the map $a \mapsto L(x,a,m)$ is convex for any $(x,m)\in \R^d\times \cP_2(\R^d)$;
  \item $A$ is bounded;
  \item   The function $b_0$ is Lipschitz continuous and has linear growth, i.e.
	\begin{equation*}
    \begin{aligned}
        |b_0(x,m) - b_0(x',m')| &\le C_b 
        \big( |x-x'| + \cW_2(m,m')\big), \\
 	|b_0(x,m)| &\le C_b  (1 + |x|  +\| m \|_1 )
    \end{aligned}
	\end{equation*}
 for all $(x,m)\in \R^d\times \cP_2(\R^d)$ and for some constants $C_b\ge0$;
  \item \label{Ass:compact.ii} $G$ and $L$ are continuous and satisfy the local Lipschitzianity and growth conditions:
  \begin{equation*}
    \begin{aligned}
    |L(x,a,m) - L(x',a,m')| &\le C_L ( 1 + |x| +  |x'| + \| m\|_1 + \| m'\|_1 ) \big(|x-x'| + \cW_2(m,m')\big), \\ 
    |G(x,m) - G(x' ,m')| &\le C_G ( 1 + |x| +  |x'| + \| m\|_1 + \| m'\|_1 ) \big(|x-x'| + \cW_2(m,m')\big), \\ 
    |L(x, a,m)| &\le C_L (1 + |x|^2 + \| m \|_2^2), \\
    | G(x, m) |& \le C_G (1 + |x|^2  + \| m \|_2^2)
  \end{aligned}
    \end{equation*}
  for all $x,x'\in \R^d$, $m,m'\in \cP_2(\R^d)$, $a\in A$ and for some constants $C_L, C_G \ge0$.
\end{enumerate}
\end{assumption}

The following result gives the convergence of discrete-time MFEs to continuous-time equilibria.
\begin{theorem}
\label{thm:compactness}
Assume that for each $k\in \N$ the multi-period MFG with data $(\xi, Z, b, L, G, k)$ admits an MFE $(\alpha^k_{t_i},m^k_{t_i})_{i=0,\dots,k}$ with state process $(X^k_{t_i})_{i=0,\dots,k}$, and denote by $\widehat \alpha^k ,\widehat X^k,\widehat m^k$ the continuous-time interpolation as in \eqref{eq:cont.interp.alpha.m} and \eqref{eq:cont.time.approx}.

Under Assumption \ref{Ass:Compact.result}, the following statements hold true:
\begin{enumerate}
    \item\label{thm:compactness:tight} 
    The sequence $(\widehat m^k )_{k\ge0}$ admits  accumulation points; i.e.,  there exists   a continuous map $m: [0,T] \to \cP_2(\R^d) $ such that, up to a subsequence,   $\widehat m^k_t  \to m_t$ in $\cW_2$ as $k\to \infty$, for all $t \in [0,T]$.
    \item\label{thm:compactness:MFE} 
    Any  accumulation point $m$ of the sequence $( \widehat m^k )_{k\ge0}$ is the law of a continuous-time MFE $(\alpha,m)$ for some Markovian control $\alpha \in \cA_M$. 
\end{enumerate}
\end{theorem}

\begin{remark}
It is worth pointing out that the choice of $A$ being bounded is mainly for simplification of the exposition. 
Indeed, for unbounded $A$, it is still possible to characterize the accumulation points $m$  of the sequence $(\widehat m^k )_{k\ge0}$, with the main difference that $(\alpha,m)$ would now be an MFE with Markovian weak control; that is, $\alpha$ is a feedback map such that the SDE \eqref{eq:SDE.cont} has a weak solution $X^\alpha$ whose law coincide with $m$, and $\alpha$ is optimal for $J_m$ (properly extended to the feedback maps for which \eqref{eq:SDE.cont} has a weak solution).
We refer to  \cite[Theorem 2.1]{Lacker15} for this notion of equilibrium.
Indeed,  the proof of Theorem \ref{thm:compactness} below can be easily be adapted to unbounded $A$. 
In doing so, the estimate \eqref{eq a priori estimate controls state measure} below can be recovered by using the argument in the proof of \cite[Lemma 5.1]{Lacker15}, and Step 4 below needs to be expanded with one additional approximation of any control by bounded ones.
Since we prefer to work under the strong formulation of the problem, the details are left to the interested reader.
\end{remark}

\subsection{Proof of Theorem \ref{thm:compactness}}

We split the proof into five steps. 
In the first two steps, we use compactness principles to show convergence of the sequences $(\widehat\alpha^k)_{k\ge1}$ and $(\widehat m^k)_{k\ge1}$.
In the ensuing two steps we show that the limiting processes $(\alpha,m)$ form MFEs.

\subsubsection{Proof of Claim \ref{thm:compactness:tight}} 
We prove each claim separately.

\medskip
\emph{Step 1: A priori estimates.}
In this step we show the estimates
  \begin{equation}\label{eq a priori estimate controls state measure}
    \sup_{k\ge1}\E\bigg[\sup_{t\in [0,T]}|\widehat X^k_t|^p \bigg] + \sup_{k\ge1} \int_{\R^d}|z|^p\widehat m^k_s(dz)ds < \infty
  \end{equation}
 for a generic $p>1$, as well as
  \begin{equation}\label{eq a priori tightness}
    \lim_{h\to0}\limsup_{k\to \infty}\sup_{\tau\in \cT_h}\E[|\widehat X^k_{\tau+h} - \widehat X^k_\tau|] =0,
  \end{equation}
  which will be used later to derive the tightness.

  Let $k\ge1$ be fixed.
  Using the linear growth condition on $b$, the boundedness of $A$ and the fact that $\xi$ and $W$ have $p^{\mathrm{th}}$ moments, we can find a constant $C>0$ that does not depend on $k$ such that
  \begin{align*}
    \E\Big[\sup_{s\in [0,t]}|\widehat X^k_s|^p\Big] &\le C\E\bigg[1 + \int_0^t |\widehat X^k_{\eta^k(s)}|^p + |\widehat\alpha_s^k|^p + \int_{\R^d}|z|^p\widehat m^k_s(dz)ds\bigg]\\
    &\le C\bigg(1 + \int_0^t \E\big[|\widehat X^k_{\eta^k(s)}|^p \big] ds + \sum_{j=0}^{k-1} \mathds{1}_{ \{t_j \leq t\} } \E\big[|\widehat X^k_{t_j}|^p\big]\delta \bigg), 
  \end{align*}
  where we have used that $\widehat m ^k_{t_j} = m ^k_{t_j} =  \P \circ (X^k_{t_j})^{-1} = 
  \P \circ (\widehat X^k_{t_j})^{-1} $, which follows from the consistency of $m^k$.
  Using the definition of $\eta^k$, we continue to obtain
  \begin{align*}
    \E\Big[\sup_{s\in [0,t]}|\widehat X^k_s|^p\Big] \le C\bigg(1 + \int_0^t\E\Big[\sup_{r \in [0,s]}|\widehat X^k_{\eta^k(r)}|^p\Big]ds \bigg)
    \le C\bigg(1 + \int_0^t\E\Big[\sup_{r \in [0,s]}|\widehat X^k_{r}|^p\Big]ds \bigg). 
  \end{align*}
Using Gronwall's lemma and the consistency of $m^k$, we obtain \eqref{eq a priori estimate controls state measure}.

We next prove \eqref{eq a priori tightness}.
Let $h>0$, and let $\cT_h$ be the set of $[0,T]$--valued stopping times $\tau$ such that $ \tau +h\le T$.
  For every $\tau\in \cT_h$, it follows by H\"older's inequality  that 
  \begin{align*}
    \E[|\widehat X^k_{\tau+h} - \widehat X^k_\tau|] 
    &\le \E\bigg[\int_{\tau}^{\tau+h}|b(\widehat X^k_{\eta^k(u)}, \widehat \alpha^k_u, \widehat m^k_u)|du \bigg] + |\sigma| \E[|Z^k_{\tau+h} - Z^k_{\tau}|]\\
    &\le h^{1/{p'}} \E\bigg[\int_0^T|b(\widehat X^k_{\eta^k(u)}, \widehat \alpha^k_u, \widehat m^k_u)|^pdu\bigg]^{1/p} + |\sigma|\sqrt{h}\\
    & \le Ch^{1/{p'}} \bigg(1 + \E\bigg[  \sup_{t\in [0,T]}|\widehat X^k_{t}|^p+ \int_0^T|\widehat \alpha^k_u|^pdu\bigg]^{1/p}\bigg) + |\sigma| \sqrt{h}\\
    &\le C h^{1/{p'}}  + |\sigma|\sqrt{h}  , 
  \end{align*}
  where $p'$ is the H\"older conjugate of $p$ and where we used  \eqref{eq a priori estimate controls state measure}.
Finally, \eqref{eq a priori tightness} easily follows from the previous estimate.
 
\medskip
  \emph{Step 2: Relaxed controls, weak limits of  $(\xi, Z^k, \widehat \alpha ^k, \widehat X ^k)$ and Skorokhod representation.}
Let $\Lambda$ denote the set of \emph{ relaxed controls} on $[0,T] \times A$; that is, the set of positive measures $\lambda$ on $[0,T] \times A$ such that $\lambda([s,t] \times A)=t-s$ for all $s,t \in [0,T]$ with $s<t$. 
Upon dividing every $\lambda \in \Lambda $ by $T$, the set $\Lambda$ corresponds to the set of probability measures on $[0,T]\times A$. 
Therefore, since $A$ is compact, $\Lambda$ is compact when endowed with the topology of weak convergence of probability measures. 
The set $\mathbb L ^p ([0,T]; A)$ is naturally included in the set of  relaxed controls via the map  $\alpha \mapsto \lambda^\alpha(dt,da):= \delta_{\alpha_t}(da)dt$. 
Moreover, any relaxed control  $\lambda$ admits (with slight abuse of notation) a factorization  $\lambda\colon  [0,T]\to \mathcal{P}(A)$ such that $\lambda(dt,da)=\lambda_t(da)dt$.

For any $k\geq 1$, define 
$$
\lambda^k(dt,da):= \delta_{\widehat \alpha _t^k}(da) dt, 
$$
notice that  $\lambda^k$ is a $\Lambda$-valued r.v.\ defined on $\Omega$ such that $\sigma\{ \lambda([0,t]\times E )\,|\,  E \in \mathcal{B}(A) \} \subset \mathcal{F}_t, \ \forall t \in [0,T].$
Since $\Lambda$ is compact, the sequence $\mathbb P \circ (\lambda^k)^{-1}$ is tight in $\mathcal P (\Lambda)$.
Let $\mathcal D $ denote the space of càdlàg functions $\varphi: [0,T] \to \R^d$,  endowed with the Skorokhod topology (see \cite{billingsley1999convergence}).
    Thanks to \eqref{eq a priori tightness}, it follows by Aldous' tightness criteria \cite[Theorem 16.9]{billingsley1999convergence} that the sequence $\mathbb P \circ (\widehat X^k)^{-1}$ is tight in  $\mathcal P (\mathcal D)$, while the sequence $\mathbb P \circ (\xi, Z^k)^{-1}$ is tight in  $\mathcal P (\mathbb R ^d \times \mathcal D)$ by Assumption \eqref{Z1}. 
Therefore, the sequence of probability measures
$$
 \mathbb P \circ (\xi, Z^k, \lambda ^k, \widehat X ^k)^{-1} 
$$
is tight in $\mathcal P (\mathbb R ^d \times \mathcal D  \times \Lambda \times \mathcal{D})$, where $\mathbb R ^d \times \mathcal D  \times \Lambda \times \mathcal{D}$ is endowed with the product topology. 
Prokhorov's theorem ensures the existence of a limiting probability measure 
$\mathbb P ^* \in  \mathcal P (\mathbb R ^d \times \mathcal D  \times \Lambda \times \mathcal{D})$ such that, {up to a subsequence} (not relabeled),
$ \mathbb P \circ (\xi, Z^k, \lambda ^k, \widehat X ^k)^{-1} \to \mathbb P ^*$ weakly as $k \to\infty$.

By Skorokhod representation theorem, we can find a probability space $(\bar \Omega, \bar{\mathcal F}, \bar{\mathbb P})$ supporting $\mathbb R ^d \times \mathcal D  \times \Lambda \times \mathcal{D}$-valued r.v.'s 
$
(\bar \xi ^k, \bar Z ^k,\bar \lambda ^k, \bar X ^k) 
$, $k \geq 1$, 
$
(\bar \xi , \bar W ,\bar \lambda , \bar X), 
$
such that: 
\begin{equation}\label{eq Skorokhod limits}
\begin{aligned}
    &\text{$\mathbb P \circ (\xi, Z^k, \lambda ^k, \widehat X ^k)^{-1} =  \bar{\mathbb P} \circ (\bar \xi ^k, \bar Z ^k,\bar \lambda ^k, \bar X ^k)^{-1}$ for any $k \geq 1$;}\\
    &\text{$\mathbb P ^* =  \bar{\mathbb P} \circ (\bar \xi , \bar W ,\bar \lambda , \bar X)^{-1}$;}\\
    &\text{$
(\bar \xi ^k, \bar Z ^k,\bar \lambda ^k, \bar X ^k)
\to
(\bar \xi , \bar W , \bar \lambda , \bar X), 
$ in $\mathbb R ^d \times \mathcal D  \times \Lambda \times \mathcal{D}$, $\bar{\mathbb P}$-a.s.}
\end{aligned}
\end{equation}
In particular, we have the limit $\widehat m_t^k \to \bar{\P} \circ (\bar X _t)^{-1}$ weakly, which together with \eqref{eq a priori estimate controls state measure}, in turn implies that
\begin{equation}
    \label{eq consistency first}
    \widehat m_t^k \to m_t \quad \text{in $\cW_2, \quad $ for $m := \bar{\P} \circ (\bar X )^{-1}$}.
\end{equation}
Moreover, 
defining the process $\bar \alpha _t:= \int_A a \bar \lambda _t (da)$, and noticing that 
$$
\begin{aligned}
  \bar{\mathbb P} \circ \bigg(\int_0^\cdot \int_A a \bar \lambda^k_s (da) ds , \bar X^k\bigg)^{-1}  &=  \mathbb P \circ \bigg(\int_0^\cdot \widehat \alpha^k_s ds ,\widehat X^k\bigg)^{-1}, \\
 \bar{\mathbb P} \circ \bigg(\int_0^\cdot \int_A a \bar \lambda_s (da) ds , \bar X \bigg)^{-1}
 &= \bar{\P} \circ \bigg(\int_0^\cdot \bar \alpha_s ds ,\bar X \bigg)^{-1},
\end{aligned}
$$ 
we have that $\bar{\mathbb P} \circ \big(\int_0^\cdot \widehat \alpha^k_s ds ,\widehat X^k \big)^{-1}$ converges weakly to $\bar{\P} \circ \big(\int_0^\cdot \bar \alpha_s ds ,\bar X \big)^{-1}$. 

We next show the continuity of $\bar X$.
Since, 
\begin{equation*}
    \bar X ^k_t = \bar \xi^k + \int_0^t \int_A a \bar \lambda _s^k(da) + b_0( \bar X ^k_s, \widehat m ^k_s)ds + \sigma \bar Z^k_t, \quad \bar{\P}\text{-a.s.,}
  \end{equation*}
 using the $\bar \P $-a.s.\ convergence of $(\bar \xi ^k, \bar Z ^k,\bar \lambda ^k, \bar X ^k)$, Gronwall's lemma and the continuity of $b_0$, we obtain
\begin{equation}
\label{eq SDE Skorokhod}
    \bar X _t = \bar \xi + \int_0^t  \bar \alpha _s + b_0( \bar X _s, m_s)ds + \sigma \bar W_t, \quad \bar{\P}\text{-a.s.}
\end{equation}
This clearly implies the continuity of $\bar X$.
Hence, the map $m: [0,T] \to \cP_2 (\R^d)$ defined in \eqref{eq consistency first} is continuous as well, which completes the proof of Claim \ref{thm:compactness:tight}. 

\subsubsection{Proof of Claim \ref{thm:compactness:MFE}} 
The proof of the claim is again divided into two steps.

\medskip
\emph{Step 3:  Markovian representation of $\alpha$ and consistency of $m$.}
  By using the 
  mimicking theorem in \cite[Corollary 3.7]{Brun-Shr13} to the SDE \eqref{eq SDE Skorokhod}, we find a probability space $( \Omega',  \cF', \P')$ along with a Brownian motion $ W'$ defined on it and a random variable $\xi'$ sharing the same law as $\bar \xi$ such that $ \P'\circ (X'_t)^{-1} = \bar \P\circ (\bar X_t)^{-1} $ for all $t \in [0, T]$, where $ X'$ satisfies
  \begin{equation*}
    X_t' = \xi' + \int_0^t\alpha_s( X_s') + b_0( X_s', m_s)ds + \sigma W'_t, \quad \P'\text{-a.s.}
  \end{equation*}
 with $ \alpha_t(x):= \E^{\bar \P}[\bar\alpha_t|\bar X_t=x]$.
Since $A$ is bounded, it follows that $\alpha_t(\cdot)$ is bounded measurable, so that by \cite[Theorem 1.1]{zhang2005strong} the previous SDE has a unique strong solution.
Hence, by uniqueness in law, we have $\P' \circ (X')^{-1} = {\P} \circ (X^\alpha)^{-1}$, where $X^\alpha$ denotes the unique strong solution  to the SDE 
\begin{equation*}
X^\alpha_t = \xi + \int_0^t\alpha_s( X_s^\alpha) + b_0( X_s^\alpha, m_s)ds + \sigma W_t, \quad \P\text{-a.s.},
\end{equation*}
defined on the initial probability space $(\Omega,\cF, \P)$.

Therefore, using the definition of $m$ in \eqref{eq consistency first}, we conclude that
$$
m_t = \bar{\P} \circ (\bar X_t)^{-1} = {\P} \circ ( X^\alpha_t)^{-1},
$$
which is the desired consistency.

\medskip
\emph{Step 4: Optimality of $\alpha$.}
In this step, we show that
\begin{equation}
    \label{eq optimality of alpha over OL controls}
    J_m(\alpha) \leq J _m(\beta), \quad \text{for any    control $\beta \in \cA_M$.}
\end{equation}
In order to do so, we would like to deduce the optimality of $\alpha$ from the optimality of $\alpha^k$ by discretizing the SDE \eqref{eq:SDE.cont} controlled by the measurable function $\beta$.
However, the Euler-Maruyama scheme from \cite{L-Yan02} for such an SDE would converge (weakly), given some regularity of the map $(t,x) \mapsto \beta_t(x)$, which is not satisfied for the generic $\beta \in \cA_M$.
We will show this result via subsequent approximations of $\beta$ through more regular Markovian controls.

\medskip
\emph{Step 4a: Optimality of $\alpha$ over  Lipschitz Markovian controls.}
We show \eqref{eq optimality of alpha over OL controls} for any $\beta \in \cA_M$ such that the map $(t,x) \mapsto \beta_t(x)$ is Lipschitz continuous.

By the continuity properties of the costs and the uniform integrability deriving from \eqref{eq a priori estimate controls state measure}, we find 
\begin{equation*}
    \begin{aligned}
          \mathbb{E}^{\bar{\P}} \bigg[ \int_0^T\int_A &L(\bar X_t,a, m_t) \bar \lambda _t(da) dt + G(\bar X _T ,m_T) \bigg] \\ 
       &= \lim_k \mathbb{E}^{\bar{\P}} \bigg[ \int_0^T\int_A L(\bar X ^k_t,a, \widehat m ^k_t) \bar \lambda ^k _t(da) dt + G(\bar X ^k_T , \widehat m ^k_T) \bigg] \\
       & = \lim_k \E\bigg[ \int_0^T \int_A L(\widehat X ^k_t, a, \widehat m ^k_t) \widehat \lambda ^k _t(da) dt+ G(\bar X ^k_T , \widehat m ^k_T) \bigg] \\
       &  = \lim_k \E\bigg[ \int_0^T L(\widehat X ^k_t,\widehat{\alpha}^k_t, \widehat m ^k_t) dt + G(\widehat X ^k_T , \widehat m ^k_T) \bigg] \\
       & = \lim_k J_{m^k}^k(\alpha^k).
    \end{aligned}
\end{equation*}
Moreover, by Jensen inequality (thus exploiting the convexity of $L$ in $a$) and the definitions of $\bar \alpha$ and the Markovian map $\alpha \in \cA_M$, we deduce
\begin{equation*}
    \begin{aligned}
       J_m(\alpha) &\leq  \mathbb{E}^{\bar{\P}} \bigg[ \int_0^T  L(\bar X_t,\bar \alpha _t, m_t)  dt + G(\bar X _T ,m_T) \bigg] \\
       &\leq  \mathbb{E}^{\bar{\P}} \bigg[ \int_0^T\int_A L(\bar X_t,a, m_t) \bar \lambda _t(da) dt + G(\bar X _T ,m_T) \bigg] . 
    \end{aligned}
\end{equation*}
Thus, combining the latter two inequalities, we conclude that
\begin{equation}
    \label{eq lowere semicont limits}
    \begin{aligned}
       J_m(\alpha) 
       &\leq\liminf_k J_{m^k}^k(\alpha^k).
    \end{aligned}
\end{equation}

We next discretize the SDE \eqref{eq:SDE.cont} controlled by $\beta$.
Let us now denote by $X^{\beta,k}$ the Euler-Maruyama discretization of $X^{\beta}$, given by
  \begin{align*}
    X_{t_i}^{\beta,k}
  &=X^{\beta,k}_{t_{i-1}} + \big({\beta^k_{t_{i-1}} + b_0(X^{\beta,k}_{t_{i-1}}, m^{k}_{t_{i-1}})}\big)\delta_k + \sigma (Z^k_{t_i} - Z^k_{t_{i-1}})
  \end{align*}
  with $\beta^{k}_{t_i}:= \beta_{t_i} (X^{\beta,k}_{t_i})$.
  We use the same notation to denote the continuous-time interpolation, i.e.,
  \begin{equation*}
    X_t^{\beta,k} = \xi + \int_0^t \beta_{\eta^k(s)}\big(X^{\beta,k}_{\eta^k(s)} \big) + b_0\big(X^{ \beta,k}_{\eta^k(s)},\widehat m^k_{\eta^k(s)} \big)ds + \sigma Z^k_t
  \end{equation*}
  with $\eta^k$ defined in \eqref{eq:cont.time.approx}.
  Since $\beta$ is assumed to be Lipschitz, we can apply \cite[Theorem 2.1]{CHAN199833} which asserts that $(X^{ \beta,k})_{k\ge1}$ converges in law to $X^{ \beta}$.
  Using the Lipschitz continuity of $\beta$, we also have that $ \beta_{\eta^k(t)}\big(X^{ \beta,k}_{\eta^k(t)} \big)$ converges to $\beta_t(X^{\beta}_t)$ for almost every $t$.
  Thus, since $L$ and $G$ are continuous, we have
  \begin{equation}\label{eq limits beta discretized}
  \begin{aligned}
       J_m(\beta) & = \E\bigg[\int_0^TL(X^{\beta}_t, \beta_t, m_t)dt+G(X^{\beta}_T,m_T)\bigg] \\
       & = \lim_{k\to \infty} \E\bigg[\sum_{j=0}^{k-1}L(X^{\beta,k}_{t_j}, \beta^{ k}_{t_j}, m^k_{t_j})\frac T k + G(X^{\beta,k}_{t_k}, m^k_{t_k})\bigg]  =  \lim_{k\to \infty} J_{m^k}^k(\beta^k).
  \end{aligned}
  \end{equation}

  This allows to deduce optimality of the candidate $\alpha$ constructed in Step 2.
  In fact, since  $(\alpha_{t_i})_{i=1,\dots,k}$ is optimal for the discrete-time game with $k$ periods, we have $$
    J_{m^k}^k(\alpha^k) \leq J_{m^k}^k(\beta^k),
$$
 for any $k\geq 1$. 
 Hence, taking limits as $k \to \infty$ and using \eqref{eq lowere semicont limits} and \eqref{eq limits beta discretized}, we deduce \eqref{eq optimality of alpha over OL controls} for any Lipschitz continuous $\beta \in \cA_M$. 
This concludes the proof of Step 4a.

\medskip
\emph{Step 4b: Optimality of $\alpha$ over  Markovian controls.}
We finally show \eqref{eq optimality of alpha over OL controls} for a (fixed)  generic $\beta \in \cA_M$.

By standard mollification, we approximate $ \beta$ by a sequence of uniformly bounded, Lipschitz-continuous functions $\beta^n:[0,T]\times \R^d \to A$, and we denote by $X^{\beta^n}$ the related state process satisfying   
\begin{equation*}
      X^{\beta^n}_t = \xi + \int_0^t \beta^n_s(X^{\beta^n}_s) + b_0(X^{\beta^n}_s, m_s)ds + \sigma W_t.
\end{equation*}
Consider the processes
  \begin{equation*}
    Y_t:= \xi + \int_0^t b_0 (Y_s, m_s) ds + \sigma W_t\quad \text{and}\quad \cE(M)_t:= \exp\Big(M_t - \frac12\langle M\rangle_t\Big)
  \end{equation*}
  for any continuous square integrable martingale $M$ with quadratic variation $\langle M \rangle$.
  Define
  \begin{equation*}
    \mathcal E ^{\beta^n}_t := \cE\Big(\int_0^\cdot\sigma^{-1} \beta^n_s(Y_s) dW_s\Big)_{t} , 
    \quad 
    \mathcal E ^\beta_t := \cE\Big(\int_0^\cdot \sigma^{-1}  \beta_s(Y_s) dW_s\Big)_{t}, 
   \end{equation*}
and probability measures $\mathbb Q ^n$, $\mathbb Q$ with corresponding densities $\mathcal E ^{\beta^n}_T $, $\mathcal E ^\beta_T$ with respect to $\P$.
By boundedness of $\beta$ and $\beta^n$, we can use
Girsanov's theorem to deduce that the processes $\tilde W ^n := W - \int_0^{\cdot}  \sigma^{-1} \beta_s^n ( Y_s) ds$ and $\tilde W :=W - \int_0^{\cdot}  \sigma^{-1} \beta_s ( Y_s) ds$ are Brownian motions with respect to $\mathbb Q ^n$ and $\mathbb Q$, respectively.

The dynamics for $Y$ can be rewritten as 
$$
Y_t:= \xi + \int_0^t \beta^n_s(Y_s) +  b(Y_s, m_s) ds + \sigma \tilde W ^n_t
$$
and, by weak uniqueness of its solution, we deduce that $\mathbb Q ^n \circ (\beta^n(Y),Y)^{-1} = \mathbb P \circ (\beta^n (X^{\beta^n}), X^{\beta^n}) ^{-1}$, which implies that
$$
J_m(\beta^n) = \mathbb E \bigg[ \int_0^T \mathcal E _t (\beta^n) L (Y_t, \beta_t^n(Y_t), m_t) dt + \mathcal E _T (\beta^n) G(Y_T, m_{T}) \bigg].
$$
Similarly, we find
$$
J_m(\beta) = \mathbb E \bigg[ \int_0^T \mathcal E _t (\beta) L (Y_t, \beta_t(Y_t), m_t) dt + \mathcal E _T (\beta) G(Y_T, m_{T}) \bigg].
$$

Since $(\beta^n)_{n\ge1}$ is uniformly bounded  and  converges to $\beta$ in $\L^1 ([0,T] \times \R^d)$, it follows by It\^o's isometry  that 
$\mathcal E ^{\beta^n} \to \mathcal E ^\beta$ in $\L^2$.
Thus, by using the continuity of $L$ and $G$ and the growth conditions, we deduce that
  \begin{align*}
  J_m(\beta) =\lim_{n\to\infty} J_m(\beta^n).
  \end{align*}
The latter limit, combined with Step 4a, implies \eqref{eq optimality of alpha over OL controls}.
This completes the proof of the theorem.

\section{Donsker-type result for MFGs: Quantitative convergence}
\label{sec:Donsker}
In this section, we present a characterization result for multi-period MFEs in which the control is a minimizer of the Hamiltonian of the problem along suitably defined adjoint processes.
While such results are classical in continuous-time MFGs and have had various applications, they have not been explored in the discrete-time case:
The characterization is thus of independent interest.
For our stability analysis, such a result will prove useful to derive the convergence from the multi-period discrete-time games to their continuous-time analogues.
In particular, with respect to the results in Section \ref{sec:convergence to continuous}, these characterizations allow us to obtain a convergence rate under the Lasry-Lions monotonicity.

In this section, we will work under Assumption \ref{Ass:Compact.result}. 
Before introducing the set-up, we comment our approach and assumptions.
\begin{remark}\label{remark separate drift}
    In this section, the key tool  is to use a Girsanov tranformation to rewrite the cost functional by decoupling the state equation from the control variable, and then to exploit the well-known  connection of the optimization problem with BSDEs.
    In the standard cases,  where the drift $b$ is assumed to be bounded (see \cite{cardelbook1}), the typical Girsanov tranformation removes the drift $b$.
    In the sequel, in order to keep the discussion concise and to allow for a better comparison with the results in Section \ref{sec:convergence to continuous}, we will work under Assumption \ref{Ass:Compact.result}, so that 
    our drift will be of separate form (i.e., $b(x,a,m)=a+b_0(x,m)$) with $A$ being bounded. 
    As a consequence, it is natural to perform the Girsanov decoupling only by removing the control term from the drift, and to use a slightly different version of a reduced Hamiltonian (with respect to the standard cases in \cite{cardelbook1}).
    We stress that all the results in this section do not directly exploit the separate form of the drift $b$, and that the case of a non-separable bounded $b$ can be treated with similar arguments.
\end{remark}

\subsection{Weak formulation of discrete-time MFGs}\label{subsec:charc.MFE.mult.period} 
{In the notation of Section \ref{sec:convergence to continuous}, take $k \in \N$ and divide the time interval $[0,T]$ into $k$ periods of length $\delta_k := \frac T k$.
In the rest of this section, $Z^k = (Z^k_{t_i})_{i=0,...,k}$ is assumed to be the discretized Brownian motion; that is, we take
$$
Z^k_{t_{i}} = W_{t_i}, \quad i=0,\dots,k.
$$}
We start by introducing the probabilistic weak formulation of the game under Assumption \ref{Ass:Compact.result}:
For $\sigma > 0$ (which is taken to be constant for simplicity), let us consider the process
\begin{equation}
\label{eq:driftless}
  Y_{t_{i+1}} =  Y_{t_i} + b_0(Y_{t_i},m_{t_i}) \delta_k + \sigma (Z^k_{t_{i+1}} - Z^k_{t_i}), \quad i=0,\dots,k-1,  \quad Y_0=\xi,
\end{equation}
and for any flow of measures $m = (m_{t_i})_{i \leq k}$ and control $\beta = (\beta_{t_i}(\cdot))_{i \leq k} \in \cA ^k_M$, we consider the process
\begin{equation}\label{eq noise Girsanov}
  Z^{\beta}_{t_i} := Z^k_{t_i} - \sum_{j=0}^{i-1}\sigma^{-1}  \beta_{t_j}(Y_{t_j}) \delta_k, \quad i=0,\dots,k-1, 
\end{equation}
{with the convention $\sum_{j=0}^{-1} P_{t_j}:= 0$, for any generic process $(P_{t_i})_i$.}
Let $\P^{\beta}$ be the probability measure on $(\Omega,\cF)$ such that 
\begin{equation}
  \label{eq:Girsanov} \P^{\beta}\circ(Z^{\beta}_{t_i},\xi)^{-1} = \P\circ(Z^k_{t_i},\xi)^{-1}, \quad  i = 0,\dots,k,
  \end{equation}
  and $Z^{\beta}$ is a $\P^{\beta}$-martingale.
  It is easily checked, using Girsanov's theorem (for the continuous-time interpolation of the given processes and where $Z^k$ interpolates a Brownian motion), that it suffices to take $\P^{\beta}$ as
  \begin{equation}
\label{eq:discrete.measu.change}
  \frac{d\P^{\beta}}{d\P} :=\exp\bigg(\sum_{j=0}^{k-1}\sigma^{-1}   \beta_{t_j}(Y_{t_j}) (Z^k_{t_{j+1}} - Z^k_{t_j}) - \frac12\sum_{j=0}^{k-1}|\sigma^{-1} \beta_{t_j}(Y_{t_j}) |^2\delta_k \bigg).
\end{equation} 
As we assume $A$ to be bounded, the probability measure $\P^{\beta}$ will be equivalent to $\P$.
The  ``weak'' cost function is given by
\begin{equation*}
  \cJ^k_m(\beta) := \E^{\P^{\beta}}\bigg[\sum_{j=0}^{k-1}L(Y_{t_j}, \beta_{t_j}(Y_{t_j}), m_{t_j})\delta_k + G(Y_{t_k}, m_{t_k}) \bigg].
\end{equation*}
By definition of $Z^{\beta}$, the process $Y$ satisfies 
  \begin{equation*}
    Y_{t_{i+1}} = Y_{t_i} + (\beta_{t_i}(Y_{t_i}) + b_0(Y_{t_i}, m_{t_i}) ) \delta_k + \sigma(Z^{\beta}_{t_{i+1}} - Z^{\beta}_{t_i}), \quad \P^{\beta}\text{-a.s}, \quad  i=0,...,k-1.
  \end{equation*}
Thus, by \eqref{eq:Girsanov}, it follows that
  \begin{equation}
  \label{eq:weak-solution}
    \P^{\beta}\circ Y_{t_i}^{-1} = \P\circ (X_{t_i}^\beta )^{-1}, \quad i=0,...,k.
  \end{equation}
This implies that
  \begin{equation}\label{eq weak cost vs strong cost}
  \begin{aligned}
    J^k_m(\beta) &= \E\bigg[ \sum_{j=0}^{k-1}L( X_{t_j}, \beta_{t_j}( X_{t_j}), m_{t_j})\delta_k + {G}(X_{t_k}, m_{t_k}) \bigg]\\
            &=\E^{\P^{\beta}}\bigg[\sum_{j=0}^{k-1}L(Y_{t_j}, \beta_{t_j}(Y_{t_j}), m_{t_j})\delta_k + G(Y_{t_k}, m_{t_k}) \bigg]=\cJ^k_m(\beta).
  \end{aligned}
  \end{equation}

 With the aim of deriving a characterization of optimal controls in the discrete-time framework, we discuss now a formulation of the problem via \emph{open-loop} controls.
 Denote by $\F^k:= (\cF^k_i)_{i=0,...,k}$ the filtration generated by $\xi$ and $Z^k$; i.e., $\cF^k_i : =\sigma (\xi, Z^k_{t_0},\cdots,Z^k_{t_i} )$.
 Now consider the set $\cA^k_{O}$ of admissible {open-loop} controls, which are just $\F^k$-progressively measurable, $A$-valued discrete processes $\gamma = (\gamma_{t_i})_{i = 0, \dots, k}$.
 For $\gamma \in \cA^k_{O}$, define $Z^\gamma$ and $\P^{\gamma}$ as in \eqref{eq noise Girsanov} and \eqref{eq:discrete.measu.change}, respectively. 
 Discrete Girsanov's theorem still holds, and the uncontrolled process $Y$ can be rewritten as
 \begin{equation}\label{eq:weakdiscreteopenloopY}
    Y_{t_{i+1}} = Y_{t_i} + (\gamma_{t_i} + b_0(Y_{t_i}, m_{t_i}) ) \delta_k + \sigma(Z^{\gamma}_{t_{i+1}} - Z^{\gamma}_{t_i}), \quad \P^{\gamma}\text{-a.s}, \quad  i=0, \dots,k-1.
  \end{equation}
With slight abuse of notation, one can define the weak cost function with respect to an open-loop control 
\begin{equation*}
    \mathcal{J}^k_m(\gamma) := \E^{\P^{\gamma}}\bigg[\sum_{j=0}^{k-1}L(Y_{t_j}, \gamma_{t_j}, m_{t_j})\delta_k + G(Y_{t_k}, m_{t_k}) \bigg].
\end{equation*}
Clearly, for a given $m$ one has $\inf_{\beta \in \cA^k_M}\cJ^k_m(\beta) \geq \inf_{\gamma \in \cA^k_O}\mathcal{J}^k_m(\gamma)$.
We have the following lemma.
\begin{lemma}
\label{lem:weak-strong}
Under Assumption \ref{Ass:Compact.result}, any MFE $(\alpha_{t_i}, m_{t_i})_{i=0,\dots, k}$ for the multi-period game is such that $\alpha$ is optimal for $\cJ^k_m$ and $\P^{\alpha}\circ(Y_{t_i})^{-1} = m_{t_i}$ for all $i=0,\dots, k$. Moreover, $$\inf_{\beta \in \cA^k_M}\cJ^k_m(\beta) = \inf_{\gamma \in \cA^k_O}\mathcal{J}^k_m(\gamma).$$
\end{lemma} 
\begin{proof}
Since $\alpha$ is optimal for $J^k_m$, by \eqref{eq weak cost vs strong cost} it is optimal for $\cJ^k_m$. Moreover, the identity $\P^{\alpha}\circ(Y_{t_i})^{-1} = m_{t_i}$ follows from \eqref{eq:weak-solution} and  the fact that $(\alpha,m)$ is an MFE. Recall that $L$ is convex in $a$. Take $\gamma \in \cA^k_O$ and define $\beta_{t_i}(y) := \E^{\P^{\gamma}}[\gamma_{t_i}|Y_{t_i} = y]$. Then it is clear from \eqref{eq:weakdiscreteopenloopY} that $\P^{\gamma} \circ Y^{-1} = \P^{\beta} \circ Y^{-1}$, as the drift is linear in the control. Therefore, from the convexity of $L$ and Jensen's inequality we have 
\begin{align*}
    \mathcal{J}^k_m(\gamma) & \ge \E^{\P^{\gamma}}\bigg[\sum_{j=0}^{k-1}L(Y_{t_j}, \beta_{t_j}(Y_{t_j}), m_{t_j})\delta_k + G(Y_{t_k}, m_{t_k}) \bigg]\\
    & = \E^{\P^{\beta}}\bigg[\sum_{j=0}^{k-1}L(Y_{t_j}, \beta_{t_j}(Y_{t_j}), m_{t_j})\delta_k + G(Y_{t_k}, m_{t_k}) \bigg] = \cJ^k_m(\beta).
\end{align*}

\end{proof}
\begin{remark}
  To be clear, Lemma \ref{lem:weak-strong} identifies the \emph{map} $\alpha_{t_i}(\cdot)$ as optimal control in both the weak and strong MFE, not the associated \emph{processes}.
  That is, the process $\alpha_{t_i}(Y_{t_i})_{i=0,\dots,k}$ is the optimal control in the weak MFE, whereas the process $\alpha_{t_i}(X_{t_i})_{i=0,\dots,k}$ is the optimal control of the strong MFE. However, when $L$ is convex, an optimal closed-loop control is also optimal among all $Z^k$-adapted processes.
\end{remark}

\subsection{BS$\Delta$Es characterization of discrete-time MFGs}
Consider the Hamiltonian $h$ and the optimized Hamiltonian $H$, defined as
  \begin{equation*}
    h(x,a,m,z) := L(x,a,m) + z  \sigma^{-1} a \quad \text{and}\quad H(x,m,z):= \inf_{a\in A}h(x,a,m,z),
  \end{equation*}
for any $(x , a , m , z )   \in \R^d \times A  \times  \cP _2 (\R^d) \times \R ^d$.
We refer the reader to Remark \ref{remark separate drift} for a comparison of our Hamiltonians with the standard cases, see \cite{cardelbook1}.

\begin{remark}
\label{Ass:Charac.gen}
We underline that, under Assumption \ref{Ass:Compact.result}, $G$ and $H$ are continuous in $(x,m)$, and satisfy, for  a constant $C>0$, 
  $$
  \begin{aligned}
      |H(x,m,z)| &\le C \big(1 + |x|^2 + \| m \|_2^2 + |z| \big), \\
      |G(x,m)| & \le C \big( 1 + |x|^2 + \| m \|_2^2 \big),  \\
      |H(x,m,z) - H(x',m,z')| &\le C ( 1 + |x| +  |x'| + \| m\|_2  ) |x-x'| + C|z-z'|,
  \end{aligned}
  $$
  for all $x,x', z,z'\in \R^d$, $m\in \cP_2 (\R^d)$.
\end{remark}
 
We now rewrite the weak cost functional $\cJ^k_m$ in terms of a suitable BS$\Delta$E.
Let $\beta\in \cA^k_O$ be any open-loop control for the multi-period MFG, and consider the solution $(\cY^\beta,\cZ^\beta,C^\beta)$ of the (linear) BS$\Delta$E
\begin{equation}\label{eq:beta.bsde}
    \cY^\beta_{t_i} = G(Y_{t_k}, m_{t_k}) + \sum_{j=i}^{k-1}h(Y_{t_j}, \beta_{t_j}, m_{t_j}, \cZ^\beta_{t_j})\delta_k - \sum_{j=i}^{k-1}\cZ^\beta_{t_j}(Z^k_{t_{i+1}} - Z^k_{t_i}) - \big(C^\beta_{t_k}- C^\beta_{t_i}\big),\quad \P\text{-a.s.}
\end{equation} 
for any $i=0,\dots,k$, with the convention $\sum_{j=k}^{k-1} P_{j_i} := 0$ for any process $(P_{t_i})_i$.
A solution to this equation is an $\R\times \R^d \times \R$-valued square integrable $\F$-adapted process $(\cY^\beta,\cZ^\beta,C^\beta)$ such that \eqref{eq:beta.bsde} holds and
 with $(C_{t_i})_{i=0,\dots,k}$ being an $\F$-martingale, with $C^\beta_0=0$, and orthogonal to $Z^k$, i.e. $\E[(C_{t_{j+1}} - C_{t_j})(Z^k_{t_{j+1}} -Z^k_{t_j})\mid \cF_{t_j}] = 0$ for all $j=0,\dots,k-1$.

Notice that, for any $m = (m_{t_i})$ and $\beta \in  \cA^k_O$,
the BS$\Delta$E \eqref{eq:beta.bsde} has linear coefficients, so that (thanks to the boundedness of $A$) its
unique solution exists by \cite[Theorem 9]{Brian-Dely-Mem02} for sufficiently large $k$. 
Since $A$ is bounded, the choice of $k$ does not depend on $\beta$, and from now on we will always assume $k$ to be large enough for the BS$\Delta$E \eqref{eq:beta.bsde} to have a unique solution.
 
Next, using \eqref{eq noise Girsanov}, we rewrite the BS$\Delta$E \eqref{eq:beta.bsde} in terms of the noise $Z^{\beta}$ as
$$
\cY^\beta_{0} = G(Y_{t_k}, m_{t_k}) + \sum_{j=0}^{k-1}L(Y_{t_j}, \beta_{t_j}, m_{t_j}) \delta_k - \sum_{j=0}^{k-1}\cZ^\beta_{t_j}(Z^{\beta}_{t_{i+1}} - Z^{\beta}_{t_i}) - C^\beta_{t_k}, \quad \P\text{-a.s.}
$$
Recall that $Z^{\beta}$ is a $\P^{\beta}$-martingale and $C^\beta_0=0$.
In addition, because $C^\beta$ is orthogonal, it is again a $\P^{\beta}$-martingale.
Hence, taking expectations with respect to  $\E^{\P^{\beta}}$ in the previous equality, we deduce that
\begin{equation}\label{eq weak cost and BSDE}
    \E[\cY^\beta_{0}] = \E^{\P^{\beta}}\bigg[G(Y_{t_k}, m_{t_k}) + \sum_{j=i}^{k-1}L(Y_{t_j}, \beta_{t_j}, m_{t_j})\delta_k \bigg]. 
\end{equation}

The following proposition characterizes discrete-time optimal controls in terms of a related BS$\Delta$E.
\begin{proposition}
\label{prop:Hamilton}
Let the flow of measures $(m_{t_i})_{i=0,\dots,k}$ be fixed.
If Assumption \ref{Ass:Compact.result} holds, then an open-loop control $(\alpha_{t_i})_{i=0,\dots,k} \in \mathcal A^k_O$ is  optimal for $\mathcal J ^k_m$ if and only if 
  \begin{equation}
  \label{eq:alpha.argmin}
    \alpha_{t_i} \in \argmin_{a\in A}h(Y_{t_i}, a, m_{t_i}, \cZ_{t_i}) \quad \text{for all}\quad  i=0,\dots,k,
  \end{equation}
   where the tuple $(\cY_{t_i}, \cZ_{t_i}, C_{t_i})_{i=0,\dots,k}$
  solves the discrete-time equation
  \begin{equation}
  \label{eq:BSDeltaE} 
    \cY_{t_i} = G(Y_{t_k}, m_{t_k}) + \sum_{j=i}^{k-1}H(Y_{t_j},  m_{t_j}, \cZ_{t_j})\delta_k - \sum_{j=i}^{k-1}\cZ_{t_j}(Z^k_{t_{j+1}} - Z^k_{t_j}) -\big(C_{t_k}- C_{t_i} \big) \quad \P\text{-a.s.}
  \end{equation}
  for all $i=0,\dots,k.$
  Moreover,  we have $ \E [ \cY_0]= \cV^m : = \inf_{\beta \in \cA_M^k} \cJ^k_m(\beta) = \inf_{\beta \in \cA_O^k} \cJ^k_m(\beta)$. 
\end{proposition}
\begin{remark}
    The equation \eqref{eq:BSDeltaE} is usually called BS$\Delta$E or backward stochastic difference equation.
    It is well-studied in the literature, especially in the context of deriving Donsker-type approximations of (continuous-time) backward stochastic differential equations, see, e.g., \cite{Brian-Dely-Mem01,Brian-Dely-Mem02,Cher-Stad13,Coh-Elli12} and references therein.
    Just as classical BSDEs, unique solutions exist for Lipschitz-continuous generators with square integrable terminal conditions (i.e., the conditions in Remark \ref{Ass:Charac.gen}) see e.g.\ \cite[Theorem 9]{Brian-Dely-Mem02}, whereas solutions may not exist (or may not be unique) for generators of super-linear growth, cf.\ \cite{Cher-Stad13}.
    It is shown in \cite[Proposition 3.2 and Theorem 4.2]{Cher-Stad13} that if the generator is of (strictly) subquadratic growth, then a unique strong solution exists.
    However, this result requires the increment of $Z^k$ to be bounded, which is not satisfied in the current setting.
\end{remark}
\begin{proof}
    Let $\alpha \in  \mathcal A^k_O$ be optimal for $\mathcal J ^k_m$ and  let $(\cY_{t_i}^\alpha, \cZ_{t_i}^\alpha, C_{t_i}^\alpha)_{i=0,\dots,k}$ be the  solution of the BS$\Delta$E \eqref{eq:beta.bsde} related to $\alpha$.
  Let $\bar\alpha$ be an $A$-valued process such that
  \begin{equation*}
    H(Y_{t_j}, m_{t_j}, \cZ_{t_j}^\alpha) = h(Y_{t_j}, \bar\alpha_{t_j}, m_{t_j}, \cZ_{t_j}^\alpha),
  \end{equation*}
  and  let $(\cY_{t_i}^{\bar{\alpha}}, \cZ_{t_i}^{\bar{\alpha}}, C_{t_i}^{\bar{\alpha}})_{i=0,\dots,k}$ be the  solution of the BS$\Delta$E \eqref{eq:beta.bsde} related to ${\bar{\alpha}}$.
  By \eqref{eq weak cost and BSDE}, Lemma \ref{lem:weak-strong} and the optimality of $\alpha$, we have $\E[\cY_0^\alpha] = \cJ^k_m(\alpha) \le  \mathcal{J}^k_m(\bar\alpha) = \E[\cY^{\bar\alpha}_0]$.
  
  We claim that actually $\cY^{\alpha}_0 = \cY^{\bar\alpha}_0$ ~$\,\P$-a.s.
  To see this, let us put $\Delta \cY :=\cY^\alpha - \cY^{\bar\alpha}$, $\Delta \cZ := \cZ^\alpha - \cZ^{\bar\alpha}, \Delta C := C^\alpha- C^{\bar\alpha} $.
    With this notation, we have
  \begin{equation}\label{eq:lipschitz.discrete.bsde}
  \begin{aligned}
    \Delta \cY_{t_i} &= \sum_{j=i}^{k-1}\Big( \sigma^{-1} \bar \alpha _{t_j} \Delta \cZ_{t_j} + h(Y_{t_j}, \alpha_{t_j}, m_{t_j}, \cZ^\alpha_{t_j}) - H(Y_{t_j}, m_{t_j}, \cZ^\alpha_{t_j})\Big)\delta _k \\
    &\quad- \sum_{j=i}^{k-1}\Delta \cZ_{t_j}(Z^k_{t_{j+1}} - Z^k_{t_j}) - \big(\Delta C_{t_k} - \Delta C_{t_i}\big).
  \end{aligned}
  \end{equation}
Let us consider the probability measure $\Q$, given by
\begin{equation*}
  \frac{d\Q}{d\P} :=\exp\bigg(\sum_{j=0}^{k-1}\sigma^{-1} \bar \alpha _{t_j} (Z^k_{t_{j+1}} - Z^k_{t_j}) - \frac12\sum_{j=0}^{k-1}|\sigma^{-1} \bar \alpha _{t_j}|^2\delta_k \bigg).
\end{equation*} 
{Since $\bar \alpha$ is bounded, we can use Girsanov's theorem and the orthogonality of $(C_{t_i})_{i=0,\dots,k}$ to take  conditional expectation in \eqref{eq:lipschitz.discrete.bsde} with respect to $\Q$. 
This yields to
  \begin{align*}
    \Delta \cY_0 = \E^{\Q}\bigg[\sum_{j=0}^{k-1}\Big(h(Y_{t_j}, \alpha_{t_j}, m_{t_j}, \cZ^\alpha_{t_j}) - H(Y_{t_j}, m_{t_j}, \cZ^\alpha_{t_j})\Big)\delta_k \Big|\cF_0 \bigg].
  \end{align*}}
  Note that $\P$ and $\Q$ are equivalent on $\cF_0$. 
  {Since  $h(Y_{t_j}, \alpha_{t_j}, m_{t_j}, \cZ^\alpha_{t_j}) - H(Y_{t_j}, m_{t_j}, \cZ^\alpha_{t_j})\ge0$ $\P$-a.s. and thus $\Q$-a.s. for all $j\ge0$}, we get $\Delta \cY_0\ge0$.
  Therefore, we have $\Delta\cY_0=0$ as claimed.

 In particular, 
 if there is $j\ge0$, $\varepsilon>0$ and $B\in \cF_{t_k}$ with $\Q(B)>0$ such that $h(Y_{t_j}, \alpha_{t_j}, m_{t_j}, \cZ^\alpha_{t_j}) - H(Y_{t_j}, m_{t_j}, \cZ^\alpha_{t_j})\ge\varepsilon$ on $B$, then we have
    $0 = \E^{\Q}[\Delta \cY_0] \ge \varepsilon\Q(B)>0$, which is a contradiction.
    Thus,  $h(Y_{t_j}, \alpha_{t_j}, m_{t_j}, \cZ^\alpha_{t_j}) - H(Y_{t_j}, m_{t_j}, \cZ^\alpha_{t_j}) = 0$ for all $j\ge0$.
  This yields that $(\cY_{t_i}^\alpha, \cZ_{t_i}^\alpha, C_{t_i}^\alpha)_{i=0,\dots,k}$ solves \eqref{eq:BSDeltaE} and that \eqref{eq:alpha.argmin} holds. 

\medskip

  Let us now prove the converse.
  Let $(\cY_{t_i}, \cZ_{t_i}, C_{t_i})_{i=0,\dots,k}$ be a solution of \eqref{eq:BSDeltaE} and assume that the pair $(\alpha_{t_i})_{i=0,\dots,k}$ satisfies \eqref{eq:alpha.argmin}.
  Then, as observed above, we have $\cJ^k_m(\alpha) = \E[\cY_0]$ and for any other control $\beta$, we have $\cJ^k_m(\beta) = \E[\cY_0^\beta]$ where $(\cY^\beta,\cZ^\beta,C^\beta)$ solves the BS$\Delta$E \eqref{eq:beta.bsde}.
  Since we have $h(x,\beta_{t_j}, m_{t_j}, z) \ge H(x,m_{t_j},z)$, it follows by the comparison theorem for BS$\Delta$E (or rather following the argument leading to $\Delta\cY_0\ge0$ in the first part of the proof) that $\cY_0\le \cY^\beta_0$, which yields that $\alpha$ is optimal.  
   
\medskip 
  Finally, the identity $ \E [ \cY_0]= \cV^m : = \inf_{\beta \in \cA_M^k} \cJ^k_m(\beta)$ follows from \eqref{eq weak cost and BSDE} and from the fact that $\alpha$ is optimal, thus completing the proof.  
\end{proof}

\subsection{BSDE characterization of continuous-time MFGs}
 Similar to Subsection \ref{subsec:charc.MFE.mult.period}, where we discussed the weak cost function in discrete-time , let us also introduce the weak cost function in continuous-time  (see also Remark \ref{remark separate drift}). 
  We refer for instance to \cite{CarmonaLacker15,possamai2021non} for in-depth discussions of probabilistic weak formulations of MFGs. 
  Introduce the state process
  \begin{equation}
      \label{eq Y continuous-time }
      Y_t := \xi + \int_0^t b_0(Y_s,m_s) ds + \sigma W_t, \quad t \in [0,T].
  \end{equation}
   Given any admissible control $\beta \in \cA_M$ and a continuous measure flow $(m_t)_t$, let us introduce the equivalent probability measure $\P^{\beta}$ with density
  \begin{equation*}
    \frac{d\P^{\beta}}{d\P} :=\exp\bigg(\int_0^T\sigma^{-1}  \beta_t(Y_t) dW_t - \frac12\int_0^T|\sigma^{-1}  \beta_t (Y_t) |^2dt\bigg).
  \end{equation*}
  The process $W^{\beta}$, defined as
  \begin{equation}
      \label{eq BM wrt alpha continuous-time }
      W^{\beta} _t := W_t - \int_0^t \sigma^{-1}   \beta _s (Y_s)  ds, \quad t \in [0,T], 
  \end{equation}
  is a $\P^{\beta}$-Brownian motion, by Girsanov's theorem.
   The weak cost function in continuous-time is
  \begin{equation*}
    \cJ_m(\beta) := \E^{\P^{\beta}}\bigg[\int_0^T L(Y_t, \beta _t(Y_t), m_t)dt + G(Y_T,m_T)\bigg].
  \end{equation*}
  Thanks to Girsanov's theorem, it follows as in Lemma \ref{lem:weak-strong} that any MFE $(\beta,m)$ for the continuous-time MFG is such that $\beta$ minimizes $\cJ_m$ and $m$ satisfies $\P^{\beta}\circ Y_t^{-1} = m_t$ for all $t\in [0,T]$.
Moreover, the MFE minimizes the Hamiltonian when $L$ is convex. Slightly more complicated than the discrete case, this characterization utilizes the mimicking theorem and the strong solvability of the state process SDE when the drift is a bounded measurable function. See \cite[Lemma 3.9]{tangpiwang2024malliavin} and \cite[Section 7]{CarmonaLacker15} for more details.
  Specifically, 
  \begin{equation}\label{eq:cont.time.weak.minimizer}
    \beta_t \in \argmin_{a\in A}h(Y_t,a, m_t,\cZ_t) \quad \P\otimes dt\text{-a.s.},
  \end{equation}
  where $(\cY, \cZ)$ and an adapted solution of the equation
  \begin{equation}
  \label{eq:BSDE.cont}
    \cY_t = G(Y_T,m_T) + \int_t^T H(Y_s, m_s, \cZ_s)ds -\int_t^T\cZ_sdW_s,\quad \P\text{-a.s.}
  \end{equation}
 and 
 \begin{equation}
     \E^{\P^{\beta}}[\cY_0] = \cV^m := \inf_{\beta \in \cA_M} \cJ_m(\beta).
 \end{equation}
 {We also note that for $\beta^k \in \cA^k_M$ and its piecewise-constant, continuous-time interpolation $\widehat \beta^k$, the probability measure $\P^{\widehat\beta^k}$ agrees with the discrete version $\mathbb{P}^{\beta^k}$ defined by \eqref{eq:discrete.measu.change} as $Z^k$ is the discretized Brownian motion.}

\subsection{Stability of equilibria: from BS$\Delta$Es to BSDEs}
The characterization proved in Proposition \ref{prop:Hamilton}  allows us to elaborate on the convergence result of Theorem \ref{thm:compactness} in terms of Donsker-type results for BSDEs. 
{To make explicit the dependence on $k$, we denote by $(Y^k_{t_i})_{i = 0, \dots, k-1}$ the process described in \eqref{eq:driftless}. 
We also define its continuous time interpolation $\widehat Y^k_t:=\sum_{i=0}^{k-1}Y^k_{t_i}\mathbf{1}_{[t_i, t_{i+1})}(t)$.}

\begin{theorem}
  Let Assumptions \ref{Ass:Compact.result}  be satisfied.
  For each $k \geq 1$, let $(\alpha^k_{t_i},m^k_{t_i})_{i=0,\dots, k}$ be an MFE for the $k$-period game and $\cV^k = \inf_{\beta \in \cA^k_M}J^k_{m^k}(\beta)$ the associated value function.
  Then, up to a subsequence, $(\cV^k)_{k\ge1}$ converges to $\cV^m:=\inf_{\beta\in \cA_M}J_m (\beta)$, the value function associated to some continuous-time equilibrium $(\alpha,m)$.
\end{theorem}
\begin{proof}
 
By Theorem \ref{thm:compactness}, if for each $k$ the pair $(\alpha^k_{t_i},m^k_{t_i})_{i=0,\dots, k}$ is an MFE for the $k$-period game, then there is a subsequence again denoted $(\alpha^k_{t_i},m^k_{t_i})_{i=0,\dots, k}$ such that the interpolation $(\widehat m^k)_k$ converges to $m$ in $\cW_2$, where $(\alpha, m)$ is an MFE for the continuous-time game for some $\alpha$. 
{Let $(\cY^k_{t_i}, \cZ^k_{t_i}, C^k_{t_i})_{i=0,\dots,k}$ denote the unique solution to the BS$\Delta$E \eqref{eq:BSDeltaE}, and $(\cY, \cZ)$ the unique solution to the BSDE \eqref{eq:BSDE.cont}.} 

Define the continuous-time interpolation $\widehat \cY^k$ of $\cY^k$ as
$$
\widehat \cY^k_t:=\sum_{i=0}^{k-1}\cY^k_{t_i}\mathbf{1}_{[t_i, t_{i+1})}(t),\quad t \in [0,T], \quad \widehat \cY ^k_T := \cY^k_{t_k}.
$$
Since the Hamiltonian $H$ is $\cW_2$-continuous in $m$, the sequence of functions $H^k(t,z):=H(\widehat Y^k_t,\widehat m^k_t,z)$ converges to $H(Y_t,m_t,z)$ pointwise. 
Similarly, $G^k:= G(\widehat Y^k_T, \widehat m^k_T)$ converges to $G(Y_T,m_T)$. Therefore, by \cite[Theorem 12]{Brian-Dely-Mem02}, it holds
\begin{equation*}
    \E\Big[\sup_{t\in[0,T]}|{\widehat \cY^k_t} - \cY_t|^2\Big] \to 0,  
\end{equation*}
as $k$ goes to infinity.
Since $\E^{\P^{\alpha^k}}[{\widehat \cY^k_0}]$ is the value function associated to the MFE $(\alpha^k_{t_i},m^k_{t_i})_{i=0,\dots, k}$ and $\E^{\P^{\alpha}}[\cY_0]$ is the value of the continuous-time game associated to $(\alpha,m)$, this concludes the proof in the light of the fact that $\P^{\alpha^k}$ and $\P^{\alpha}$ agree on $\cF_0$. 
\end{proof}

\subsection{Convergence rates under additional  Lasry-Lions monotonicity}
In this subsection, we use the characterization of Proposition \ref{prop:Hamilton} in order to obtain a convergence rate under the Lasry-Lions monotonicity condition.
To this end, we enforce the following sufficient conditions.
\begin{assumption}
  \label{ass.rate.nonCLL}
    \item[(i)] Assumption \ref{Ass:Compact.result} holds.
    \item[(ii)] $b_0$ does not depend on $m$ and Assumption \ref{ass:LL.condition} holds with $F,G$ satisfying the Lasry-Lions monotonicity (cf.\ Definition \ref{Def:LL}); 
    \item[(iii)] There is a constant $C_h>0$ such that the reduced Hamiltonian $h_0(x,a,z):=L_0(x,a) + z \sigma^{-1} a $ is $C_h$-strongly convex; i.e., for every $x,z$, the function $a\mapsto h_0(x,a,z) -\frac{C_h}{2}|a|^2$ is convex.
\end{assumption}

Note that using the growth condition of $b_0$ and Gronwall's inequality, one can obtain the approximation estimate
 \begin{equation}
     \label{eq estimate Yk to Y rate}
  \begin{aligned}
    \E\bigg[\sup_{t \in [0,T] } |Y_t-\widehat Y^k_t|^2 \bigg] & \le C\delta_k, \quad 
    \E\bigg[\sup_{t \in [0,T] }( |Y_t-\widehat Y^k_t|^4 + |Y_t-\widehat Y^k_{\eta^k(t)} |^4) \bigg]  \le C\delta_k^2, 
      \end{aligned}
  \end{equation}
  that will be used several times in the sequel.

The following theorem is the main result of this section, establishing a rate of convergence for the discretization of MFGs.
\begin{theorem}
\label{thm:BSDE.rate.LL}
  Let Assumption \ref{ass.rate.nonCLL} be satisfied.
  For each $k\ge1$, let $(\alpha^k_{t_i}, m^k_{t_i})_{i=0,\dots,k}$ be the MFE for the $k$-period MFG with data $(\xi,W, b, L, G,k)$, with  associated state process $ X^k$. 
  Denote by $\widehat X^k$ the continuous-time interpolation as in \eqref{eq:cont.time.approx} and define the continuous-time interpolation feedback  $\widehat \alpha^k _t (x) := \sum _{i=0}^{k-1} \alpha^k _{t_i} (x) \mathbf 1 _{[t_i,t_{i+1})}(t)$, $x \in \R^d$.
  Let $(\alpha,m)$ be the MFE for the continuous-time MFG, with $X$ denoting the associated state process. 
  If, for a constant $K_\alpha\geq0$ one has 
  $$
  |\alpha _{\bar t} (\bar x) - \alpha _{ t} ( x)| \leq K_\alpha \big( |\bar t - t|^{\frac12} + |\bar x - x| \big), 
  \quad  \text{for all} \quad \bar t, t \in [0,T], \ \bar x,x \in \R^d ,
  $$ 
  then we have the rates of convergence
    \begin{equation}
    \label{eq:estima.alpha.LL}
      \Big(\E^{\P^{\alpha}} + \E^{\P^{\alpha^k}}\Big)\bigg[\int_0^T|\widehat\alpha_t^k( \widehat Y _t^k) - \alpha_t(Y_t)|^2dt\bigg] \le \frac{C}{\sqrt{k}},
  \end{equation}  
 and
    \begin{equation}
    \label{eq:estima.X.LL}
      \bigg( \E\Big[\sup_{t\in [0,T]}|\widehat X^k_t - X_t|^{2} + \int_0^T|\widehat \alpha^k_t(\widehat X^k_t) -\alpha_t(X_t) |^2dt\Big] \bigg)^{\frac 12} \le \frac{C}{\sqrt{k}},
  \end{equation}
  for all $k$ and for some constant $C>0$ that does not depend on $k$.
\end{theorem}

\begin{remark}
    It is shown for instance in \cite{Card17,chassagneux2014probabilistic} that under the Lasry-Lions monotonicity condition, and additional regularity assumptions on the coefficients, the optimal control $\alpha$ is differentiable in space.
    Alternatively, the Lipschitzianity of $\alpha$ can be shown by studying the regularity of the related forward-backward system of SDEs:
    For explicit conditions, we refer to Assumption 5.b in \cite{dianetti2022strong} (in combination with Remark 2.3 and the proof of Lemma 2.5 therein). Similarly, the Hölder continuity in time of the feedback control is also a standard consequence of other regularity conditions on the coefficients. See \cite[Remark 1.12]{achdou2021mfgbook} for a discussion. For alternative conditions, see also Remark \ref{remark:alternativeconditions}.
\end{remark}

\begin{proof}
We prove the two estimates in the statements separately. 
\smallbreak\noindent
\emph{Proof of \eqref{eq:estima.alpha.LL}.} The proof is divided into three steps.
\\ \noindent
\emph{Step 1.} In this step, we rewrite the costs $ J _{m^k}^k (\alpha^k)$ and $J _{m} (\alpha)$ in a suitable way. 

Recall the uncontrolled process $ Y^k $ defined in \eqref{eq:driftless} and the continuous-time interpolations $ \widehat Y^k $, $ \widehat m^k$ and $\widehat \alpha^k$ (as in the statement).
Using Proposition \ref{prop:Hamilton}, we have that the optimal control $\alpha^k_{t_i}(Y^k_{t_i})$ satisfies
\begin{equation*}
    \alpha_{t_i}^k (Y^k_{t_i}) \in \argmin_{a\in A}h(Y_{t_i}^k, a, m_{t_i}^k, \cZ^k_{t_i}), \quad  i=0,\dots,k,
  \end{equation*}
   where the tuple $(\cY^k_{t_i}, \cZ^k_{t_i}, C^k_{t_i})_{i=0,\dots,k}$
  satisfies
  \begin{equation*}
    \cY^k_{t_i} = G(Y_{t_k}^k, m^k_{t_k}) + \sum_{j=i}^{k-1}H(Y^k_{t_j},  m^k_{t_j}, \cZ ^k_{t_j})\delta_k - \sum_{j=i}^{k-1}\cZ^k_{t_j}(W_{t_{j+1}} - W_{t_j}) -\big( C_{t_k}^k - C_{t_i}^k\big).
  \end{equation*}
  Using the definition of $W^\alpha$ in \eqref{eq BM wrt alpha continuous-time }, we rewrite the previous BS$\Delta$E for the continuous-time interpolation $( \widehat\cY ^k ,  \widehat \cZ ^k,  \widehat C^k )$ as
  \begin{equation*}
  \begin{aligned}
           \widehat \cY^k_{0} =& G( \widehat Y_{T}^k,  \widehat m^k_{T}) + \int_0^T \big( H(\widehat Y^k_{s}, \widehat m^k_{s}, \widehat \cZ^k_{s}) +  \widehat \cZ^k_{s} \sigma^{-1}  \alpha_s (Y_s) \big)d s  - \int_0^T \widehat \cZ^k_{s} dW^\alpha_s  -  \widehat C_{T}^k.
    \end{aligned}
  \end{equation*}
  Hence, using that $\P^{\alpha^k}$ and $\P^\alpha$ agree on $\cF_0$ and that $W^\alpha$ is a $\P^\alpha$-Browian motion, the optimal cost of the $k$-period game takes the form

\begin{align*}
    J^k_{m^k}(\alpha^k) &= \E^{\P^{\alpha^k}} [ \cY_0^k ] =  \E^{\P^{\alpha}} [ \widehat \cY_0^k ]\\
    & =\E^{\P^{\alpha}}\bigg[ \int_0^TL(\widehat Y^k_{s}, \widehat\alpha^k_{s}(\widehat Y ^k_{s}), \widehat m^k_{s})ds -  \widehat C_{T}^k\\
    &\qquad \quad + \int_0^{T}  \widehat \cZ^k_{s} \sigma^{-1}\big(  \alpha_s(Y_s) -  \widehat\alpha^k_{s}(\widehat Y^k_{s} )\big)ds + G(\widehat Y^k_{T}, \widehat m^{k}_T) \bigg].
  \end{align*}
  In a similar way, by manipulating the BSDE for $(\cY,\cZ)$ in \eqref{eq:BSDE.cont}, the value of the continuous-time game can be written as
  \begin{align*}
    J_m(\alpha)  & = \E^{\P^{\alpha^k}}\bigg[\int_0^TL(Y_s,\alpha_s(Y_s), m_s)ds\\
    &\qquad \quad + \int_0^T\cZ_{s} \sigma^{-1}\big(  \alpha_s(Y_s)- \widehat\alpha^k_{s}(\widehat Y^k_{s}) \big) ds + G(Y_T, m_T) \bigg].
  \end{align*}
  Therefore, we can compute the difference $J_m(\alpha) - J^k_{m^k}(\alpha^k)$ under $\P^\alpha$ and $\P^{\alpha^k}$, obtaining
  \begin{align*}
    &J_m (\alpha) - J^k_{m^k}(\alpha^k)\\
    & = \E^{\P^{\alpha}}\bigg[\int_0^T L( Y_t, \alpha_t(Y_t), m_t) - L( \widehat Y^k_{t}, \widehat\alpha^k_{t}( \widehat Y^k_{t}), \widehat m_t^k)dt\\
    &\quad +  \int_0^T\widehat \cZ_{t}^k \sigma^{-1}\big(  \alpha_s(Y_s) - \widehat\alpha^k_{s}(  \widehat Y^k_{s}) \big)ds + G(Y_T, m_T) - G( \widehat Y^k_{T}, \widehat m^{k}_T) \bigg] + \E^{\P^\alpha}[ \widehat C^k_{T}]
  \end{align*}
  and 
  \begin{align*}
    &J_m(\alpha) - J^k_{m^k}(\alpha^k)\\
      & = \E^{\P^{\alpha^k}}\bigg[ \int_0^TL(Y_t, \alpha_t(Y_t), m_t) - L(  \widehat Y^k_{t}, \widehat\alpha^k_{t}(  \widehat Y^k_{t}), \widehat m_t^k)dt\\
      &\quad + \int_0^T \cZ_{t} \sigma^{-1}\big(  \alpha_s(Y_s) -  \widehat\alpha^k_{s}(  \widehat Y^k_{s})  \big)ds + G(Y_T, m_T) - G(  \widehat Y^k_{T}, \widehat m^{k}_T) \bigg].
  \end{align*}
  Subtracting these two expressions and using the fact that the cost $L$ is separated gives
  \begin{equation}\label{eq:Before.LL}
  \begin{aligned}
    0 &= \big(J_m(\alpha) - J^k_{m^k}(\alpha^k)\big) - \big(J_m(\alpha) - J^k_{m^k}(\alpha^k)\big)\\
    & = \E^{\P^{\alpha}}\Big[G(Y_T, m_T) - G(\widehat Y_T^k, \widehat m_T^k) \Big] - \E^{\P^{\alpha^k}}\Big[G(Y_T, m_T) - G(\widehat Y_T^k, \widehat m_T^k) \Big]+ \E^{\P^\alpha}[ \widehat C^k_{T}]\\
    &+\E^{\P^{\alpha}}\bigg[\int_0^TF(Y_t, m_t) - F(\widehat Y_t^k, \widehat m_t^k)dt \bigg] - \E^{\P^{\alpha^k}}\bigg[\int_0^TF(Y_t, m_t) - F(\widehat Y_t^k, \widehat m_t^k)dt \bigg]\\
    &+ \E^{\P^{\alpha}}\bigg[\int_0^T\big( L_0(Y_t, \alpha_t(Y_t)) + \widehat\cZ^k_t \sigma^{-1}  \alpha_t(Y_t)  \big)\\
    &\qquad\qquad - \big( L_0(\widehat Y_t^k, \widehat\alpha^k_t( \widehat Y^k_t)) + \widehat\cZ^k_t \sigma^{-1}  \widehat\alpha_t^k(  \widehat Y_t^k)  \big)dt\bigg] \\
    & - \E^{\P^{\alpha^k}}\bigg[\int_0^T\big( L_0(Y_t, \alpha_t(Y_t)) + \cZ_t \sigma^{-1} \alpha_t(Y_t)  \big)\\
    &\qquad \qquad - \big( L_0(\widehat Y_t^k, \widehat\alpha^k_t( \widehat Y^k_t)) + \cZ_t \sigma^{-1}  \widehat\alpha_t^k(  \widehat Y_t^k)  \big)dt\bigg].
  \end{aligned}
  \end{equation}
  \smallbreak\noindent
\emph{Step 2.}
In this step, we further manipulate \eqref{eq:Before.LL} using the Lasry-Lions monotonicity of the data and the convexity of $h_0$.

  Observe that, using the Lasry-Lions monotonicity condition for $G$, it holds
  \begin{align*}
    &\E^{\P^{\alpha}}\Big[G(Y_T, m_T) - G(\widehat Y_T^k, \widehat m_T^k) \Big] - \E^{\P^{\alpha^k}}\Big[G(Y_T, m_T) - G(\widehat Y_T^k, \widehat m_T^k) \Big]\\
    &= \E^{\P^{\alpha}}\Big[G(Y_T, m_T) - G(Y_T, \widehat m^k_T)\Big] - \E^{\P^{\alpha^k}}\Big[G(\widehat Y^k_T, m_T) - G(\widehat Y^k_T, \widehat m^k_T) \Big]\\
    &\quad +\E^{\P^\alpha}\Big[G(Y_T, \widehat m^k_T) - G(\widehat Y_T^k, \widehat m_T^k) \Big] + \E^{\P^{\alpha^k}}\Big[G(\widehat Y^k_T, m_T) - G( Y_T,  m_T) \Big]\\
    &\ge \E^{\P^\alpha}\Big[G(Y_T, \widehat m^k_T) - G(\widehat Y_T^k, \widehat m_T^k) \Big] + \E^{\P^{\alpha^k}}\Big[G(\widehat Y^k_T, m_T) - G( Y_T,  m_T) \Big].
  \end{align*}
  Similarly, using the Lasry-Lions monotonicity condition for $F$, we have
  \begin{align*}
    &\E^{\P^{\alpha}}\bigg[\int_0^TF(Y_t, m_t) - F(\widehat Y_t^k, \widehat m_t^k)dt \bigg] - \E^{\P^{\alpha^k}}\bigg[\int_0^TF(Y_t, m_t) - F(\widehat Y_t^k, \widehat m_t^k)dt \bigg]\\
    &\ge \E^{\P^\alpha}\bigg[\int_0^TF(Y_t, \widehat m^k_t) - F(\widehat Y_t^k, \widehat m_t^k)dt \bigg] + \E^{\P^{\alpha^k}}\bigg[\int_0^TF(\widehat Y^k_t, m_t) - F( Y_t,  m_t)dt \bigg].
  \end{align*}

  Let us now set
  \begin{equation*}\label{eq estimate E1}
  \begin{aligned}
    E_1^k &:= \E^{\P^{\alpha}}\bigg[\int_0^TF(Y_t, \widehat m^k_t) - F(\widehat Y^k_t, \widehat m^k_t)dt\bigg] + \E^{\P^{\alpha^k}}\bigg[\int_0^TF( \widehat Y ^k_t,   m_t) - F(Y_t, m_t)dt\bigg]\\
    &\quad + \E^{\P^{\alpha}}\bigg[G(Y_T, \widehat m^k_T) - G(\widehat Y^k_T, \widehat m^k_T)\bigg] + \E^{\P^{\alpha^k}}\bigg[G( \widehat Y ^k_T,  m _T) - G(Y_T, m_T)\bigg]. 
  \end{aligned}
  \end{equation*}
  {Recall the $C_h$-convexity condition for $h_0$ in $a$, and the fact that $\alpha_t(Y_t)$ and $\widehat\alpha^k_{t_i} (\widehat Y_{t_i}^k)$, respectively, minimize $a\mapsto h_0(Y_t, a, \cZ_t)$ and $a\mapsto h_0 (\widehat Y^k_{t_i}, a, \widehat\cZ^k_{t_i})$. 
  These conditions imply $|a - \widehat\alpha^k_t(\widehat Y^k_t)|^2 \le \frac{4}{C_h}(h_0(\widehat Y^k_t, a, \widehat \cZ^k_t) - h_0(\widehat Y^k_t, \widehat\alpha^k_t(\widehat Y^k_t), \widehat \cZ^k_t))$ for every $a\in A$ (see e.g. \cite[Lemma 5.2]{Car-Tan-Zha24} for details.)} 
  {The same argument applies to $\alpha_t(Y_t)$.}
  
  Coming back to \eqref{eq:Before.LL}, using  the latter inequalities and the definition of $E_1^k$, we obtain 
   \begin{equation}\label{eq estimate ultimate rate of convergence}
    \begin{aligned}
    0 & \ge E^k_1 + \E^{\P^\alpha}[ \widehat C^k_{T}]
    +\E^{\P^{\alpha}}\bigg[\int_0^Th_0(Y_t, \alpha_t(Y_t), \widehat\cZ^k_t) - h_0(\widehat Y_t^k, \widehat\alpha^k_t(\widehat Y^k_t), \widehat\cZ^k_t)dt\bigg]\\ 
    & - \E^{\P^{\alpha^k}}\bigg[\int_0^Th_0(Y_t, \alpha_t(Y_t), \cZ_t) - h_0(\widehat Y_t^k, \widehat\alpha^k_t( \widehat Y^k_t) ,\cZ_t)dt \bigg]\\
    &= E^k_1 + \E^{\P^\alpha}[ \widehat C^k_{T}]
    +\E^{\P^{\alpha}}\bigg[\int_0^Th_0(Y_t, \alpha_t(Y_t), \widehat\cZ^k_t) - h_0(\widehat Y^k_t,\alpha_t(Y_t), \widehat \cZ^k_t )dt\bigg]\\
    &\quad  + \E^{\P^{\alpha}}\bigg[\int_0^T h_0(\widehat Y^k_t,\alpha_t(Y_t), {\widehat\cZ^k_t})- h_0(\widehat Y_t^k, \widehat\alpha^k_t(\widehat Y^k_t), \widehat\cZ^k_t)dt\bigg]\\ 
    &\quad - \E^{\P^{\alpha^k}}\bigg[\int_0^Th_0(Y_t, \alpha_t(Y_t), \cZ_t) - h_0(Y_t, \widehat\alpha_t^k(\widehat Y_t^k), \cZ_t)dt\bigg]\\
    &\quad - \E^{\P^{\alpha^k}}\bigg[\int_0^T h_0(Y_t, \widehat \alpha_t^k(\widehat Y_t^k), \cZ_t)   - h_0(\widehat Y_t^k, \widehat\alpha^k_t(\widehat Y^k_t) ,\cZ_t)dt \bigg]\\
    &\ge E^k_1 + \E^{\P^\alpha}[ \widehat C^k_{T}] + E^k_2 + \frac{C_h}{4}\Big(\E^{\P^{\alpha}} + \E^{\P^{\alpha^k}}\Big)\bigg[\int_0^T|\widehat\alpha_t^k(\widehat Y_t^k) - \alpha_t(Y_t)|^2dt\bigg],
  \end{aligned}
  \end{equation}
  where $E^k_2$ is given by
  \begin{align*}
    E_2^k: &= \E^{\P^{\alpha}}\bigg[\int_0^Th_0(Y_t, \alpha_t(Y_t), \widehat\cZ^k_t) - h_0(\widehat Y_t^k,  \alpha _t(Y _t), \widehat\cZ^k_t)dt\bigg]\\
    &\quad - \E^{\P^{\alpha^k}}\bigg[\int_0^T h_0(Y_t, \widehat \alpha_t^k( \widehat Y_t^k), \cZ_t)   - h_0(\widehat Y_t^k, \widehat\alpha^k_t(\widehat Y^k_t) ,\cZ_t)dt \bigg].
  \end{align*}
  
 \smallbreak\noindent
\emph{Step 3.} 
In this step, we estimate the terms  $ E^k_1, \E^{\P^\alpha}[  \widehat C^k_{T}] , E^k_2$ in \eqref{eq estimate ultimate rate of convergence} in order to conclude the proof.
  To estimate $E_1^k$, observe that, by the (local) Lipschitz-continuity condition of $F$ and $G$, we have
  \begin{align*}
    |E_1^k| & \le C\E^{\P^\alpha}\bigg[\int_0^T(|Y_t| + |\widehat Y^k_t|)|Y_t - \widehat Y^k_t|dt  \bigg] + C\E^{\P^{\alpha^k}}\bigg[\int_0^T(|Y_t| + |\widehat Y^k_t|)|Y_t - \widehat Y^k_t|dt  \bigg]\\
    &\quad + C\E^{\P^\alpha}\bigg[(|Y_T| + |\widehat Y^k_T|)|Y_T - \widehat Y^k_T| \bigg] + C\E^{\P^{\alpha^k}}\bigg[(|Y_T| + |\widehat Y^k_T|)|Y_T - \widehat Y^k_T|\bigg]. 
      \end{align*}
  Thus, using   \eqref{eq estimate Yk to Y rate} and Hölder's inequality gives
  \begin{equation}
      \label{eq estimate E1 second}
      \begin{aligned}
     | E_1^k|  \leq &  C \E \Big[ \big( ( \cE ^{\alpha}_T) ^2 + ( \cE ^{\alpha^k}_T)^2 \big) \sup_{t\in[0,T]}\big( | Y_t|^2 +| \widehat Y_t^k|^2 \big)\Big]^{\frac12} \E\bigg[\int_0^T|Y_t -\widehat Y^k_t|^2dt +|Y_T -\widehat Y_T^k| ^2 \bigg]^{\frac12}\\ 
      & \leq C \sqrt{\delta_k},     
      \end{aligned}
  \end{equation}
  where the last inequality uses \eqref{eq estimate Yk to Y rate}, the boundedness of $A$, and the fact that $W$ and $\xi$ have moments of order $4$.

  Concerning $E_2^k$, we also use the local Lipschitz property of $h_0$ in $x$ and H\"older inequality as above to arrive at
    \begin{equation}
      \label{eq estimate E2 second}
      \begin{aligned}
     | E_2^k|  \leq &  C \E \Big[ \big( ( \cE ^{\alpha}_T) ^2 + ( \cE ^{\alpha^k}_T)^2 \big) \sup_{t\in[0,T]}\big( | Y_t|^2 +|\widehat Y_t^k|^2 \big)\Big]^{\frac12} \E\bigg[\int_0^T|Y_t -\widehat Y^k_t|^2dt  \bigg]^{\frac12} 
       \leq C \sqrt{\delta_k}.      
      \end{aligned}
  \end{equation}

We finally move to the estimation of $\E^{\P^{\alpha}}[ \widehat C^k_{T}]$.
To this end, we will discretize the feedback $\alpha$ via a sequence of maps $(\tilde \alpha^k_{t_i} ) _{i=0,...,k}$, {where $\tilde \alpha^k_{t_i} (x) = \alpha_{t_i}(x), \ i =0, \dots, k$.}
Alternatively, we may write $\tilde \alpha^k_t (\cdot) = \alpha_{\eta^k(t)}(\cdot)$.
To simplify the notation, in this final part of the proof we will assume (without loss of generality) $\sigma$ to be the identity matrix.
Define the $\mathbb{F}^k$-measurable process $\cE ^{ \tilde \alpha ^k} = (\cE ^{ \tilde \alpha ^k}_{t_i})_{i=0,...,k}$ as the solution to the equation
  \begin{equation*}
      \cE^{\tilde \alpha^k}_{t_i} = 1 + \sum_{j=0}^{i-1}\tilde \alpha ^k_{t_{j}} ( Y^k_{t_j})\cE^{\tilde \alpha^k}_{t_j}(W_{t_{j+1}} - W_{t_j}), \quad i=0, \dots, k.
\end{equation*}
Using the fact that $C^k$ is a $\P$-martingale with $C^k_{t_0} = 0$ and which is orthogonal to $(W_{t_i})_{i=0,...,k}$, with $T = t_k$ we find
  \begin{equation}\label{eq expected k CT 0}
    \E [  \cE^{\tilde \alpha^k}_{T}  \widehat C_T^k]  = \E \bigg[C_{t_k}^k + \sum_{j=0}^{k-1}\tilde \alpha ^k_{t_j }\cE^{\tilde \alpha^k}_{t_j}(W_{t_{j+1}} - W_{t_j})C^k_{t_k} \bigg] = 0.
    \end{equation}
Introduce the linear continuous-time interpolation $\widehat \cE ^{ \tilde \alpha ^k}$ of the process $\cE ^{ \tilde \alpha ^k}$ as the stochastic integral
$$
 \widehat \cE ^{ \tilde \alpha ^k}_t = 1+ \int_0^t \tilde \alpha ^k_{\eta^k(t)} ( \widehat Y^k_{\eta^k(t)})  \cE^{\tilde \alpha^k}_{\eta^k(t)} d W_t, \quad t \in [0,T],
$$
and notice that $\widehat \cE ^{ \tilde \alpha ^k}_{t_i} = \cE ^{ \tilde \alpha ^k}_{t_i},$ for $i=0,\dots,k$.
Hence, it follows that $\widehat \cE ^{ \tilde \alpha ^k}_{\eta^k(t)} = \cE ^{ \tilde \alpha ^k}_{\eta^k(t)}$ for any $t \in [0,T]$.
Notice also that, by the boundedness of $A$ (thus, of the moments of $ \widehat \cE ^{ \tilde \alpha ^k}$),  for any $t \in [0,T]$, we have
\begin{equation}\label{eq E linear vrcadlag}
\begin{aligned} 
\E \Big[ \big|\widehat \cE ^{ \tilde \alpha ^k}_t - \widehat \cE ^{ \tilde \alpha ^k}_{\eta^k(t)} \big|^2 \Big]
& = \E \bigg[ \bigg| \int_{\eta^k(t)}^t  \tilde \alpha ^k_{\eta^k(s)} ( \widehat Y^k_{\eta^k(s)}) \widehat \cE^{\tilde \alpha^k}_{\eta^k(s)} d W_s  \bigg|^2 \bigg] \\
& =   \int_{\eta^k(t)}^t \E \Big[ \big|  \tilde \alpha ^k_{\eta^k(s)} ( \widehat Y^k_{\eta^k(s)}) \widehat \cE^{\tilde \alpha^k}_{\eta^k(s)} \big|^2 \Big] ds \leq C \delta_k.
\end{aligned}
\end{equation}
Recall that the process $(\cE^{\alpha}_t)_{t\in [0,T]}$ solves the SDE
  \begin{equation*}
    \cE^{\alpha}_t = 1 + \int_0^t  \alpha_s \cE^{ \alpha }_sdW_s,
  \end{equation*}
so that that $\widehat \cE^{\tilde\alpha^k}$ is nothing but a discretization of $\cE^\alpha$.

We now proceed with the estimation of $\E^{\P^{\alpha}}[ \widehat C^k_{T}]$.
Thanks to \eqref{eq expected k CT 0}, by H\"older's inequality, we find
  \begin{align*}
    \E^{\P^\alpha}[ \widehat C^k_{T}] & = \E [(\cE^{\alpha}_T -\widehat \cE ^{\tilde\alpha^k}_T) \widehat C^k_{T}] + \E  [\cE ^{\tilde\alpha^k}_T  \widehat C_T^k]\\
    & \le \E\big[\big|\cE^{\alpha}_T - \widehat \cE ^{\tilde\alpha^k}_T \big|^2\big]^{\frac12}\E\big[| \widehat  C^k_T|^2 \big]^{\frac12}\le \E\big[\big|\cE^{\alpha}_T - \widehat \cE ^{\tilde\alpha^k}_T \big|^2\big]^{\frac12},
  \end{align*}
where the latter inequality follows from the fact that $\sup_k\E[| \widehat C^k_T|^2]<\infty$, since $( \widehat C^k_T)_k$ converges to zero in $L^2$, see e.g.\ \cite{Brian-Dely-Mem02}. 
Next, we have
\begin{align*}
    \E\big[|\cE^{\alpha}_t - \widehat \cE^{\tilde\alpha^k}_t \big|^2] 
    & = \E \bigg[ \bigg| \int_0^t \Big( \alpha_s(Y_s) \cE^{\alpha}_s - \alpha_{\eta^k(s)} (\widehat Y ^k _{\eta^k(s)}) \widehat \cE ^{\tilde\alpha^k}_{\eta^k(s)}     \Big) d W_s \bigg|^2 \bigg] \\
    & = \E \bigg[  \int_0^t \Big| \alpha_s(Y_s) \cE^{\alpha}_s - \alpha_{\eta^k(s)} (\widehat Y ^k _{\eta^k(s)}) \widehat \cE ^{\tilde\alpha^k}_{\eta^k(s)} \Big|^2 d s \bigg]\\
    & \leq  C \int_0^t \E \Big[ \big| \cE^{\alpha}_s - \widehat \cE ^{\tilde\alpha^k}_{\eta^k(s)} \big|^2 \Big] 
    + \E \Big[ \big| \widehat \cE ^{\tilde\alpha^k}_{\eta^k(s)} \big|^2 \big| \alpha_s(Y_s) - \alpha_{\eta^k(s)} (\widehat Y ^k _{\eta^k(s)})  \big|^2 \Big]d s \\
    & \leq  C \int_0^t \E \Big[ \big| \cE^{\alpha}_s - \widehat \cE ^{\tilde\alpha^k}_{\eta^k(s)} \big|^2 \Big] 
    + \Big( \E \Big[ \big| \alpha_s(Y_s) - \alpha_{\eta^k(s)} (\widehat Y ^k _{\eta^k(s)})  \big|^4 \Big] \Big) ^{\frac12} d s.
\end{align*}
Using \eqref{eq E linear vrcadlag}, the regularity of $\alpha$, and then \eqref{eq estimate Yk to Y rate},  we find
\begin{equation}\label{label:alpharegularity}
    \begin{split}
    \E\big[|\cE^{\alpha}_t - \widehat \cE^{\tilde\alpha^k}_t \big|^2] 
    & \leq  C \bigg( \delta_k + \int_0^t \E \Big[ \big| \cE^{\alpha}_s - \widehat \cE ^{\tilde\alpha^k}_{s} \big|^2 \Big] 
    + \Big(  | s -  \eta^k(s) \big|^2 +\E \big[ | Y_s - \widehat Y ^k _{\eta^k(s)} | ^4 \big] \Big) ^{\frac12} d s \bigg) \\
    & \leq  C \bigg( \delta_k + \int_0^t \E \Big[ \big| \cE^{\alpha}_s - \widehat \cE ^{\tilde\alpha^k}_{s} \big|^2 \Big] ds \bigg).
    \end{split}
\end{equation}
Thanks to Gronwall's inequality, we conclude that
 $\E\big[|\cE^{\alpha}_T - \widehat \cE^{\tilde\alpha^k}_T \big|^2] \leq C \delta_k$, which in turn gives 
 $$ 
 \E^{\P^\alpha}[ \widehat C^k_{T}] \leq C \delta_k.
 $$ 
Combining this with \eqref{eq estimate E1 second} and \eqref{eq estimate E2 second}, it follows from \eqref{eq estimate ultimate rate of convergence} that \eqref{eq:estima.alpha.LL} holds.
 \smallbreak\noindent
\emph{Proof of \eqref{eq:estima.X.LL}.}
    Recall the notation
  \begin{equation*}
    X^k_{t_i} = \xi + \sum_{j=0}^{k-1} ( b_0 (X^k_{t_j}) + \alpha^k_{t_j}(X^k_{t_j})) \delta_k +\sigma W_{t_i}
    \quad \text{and}\quad     
    X_t = \xi +\int_0^t b_0(X_s)  + \alpha_s(X_s) ds + \sigma W_t
  \end{equation*}
    for the discrete-time and continuous-time state processes, respectively, and also recall the continuous-time interpolation $\widehat X^k$ of $X^k$ given in \eqref{eq:cont.time.approx}.
    Using Lipschitz-continuity of $b$ and Gronwall's inequality, we directly obtain
    \begin{align*}
      \E\Big[\sup_{t\in [0,T]}|\widehat X^k_t - X_t|^2\Big] &\le C\E\bigg[\int_0^T|\widehat \alpha^k_t(\widehat X^k_t) - \alpha_t(X_t)|^2dt\bigg] + C {\delta_k}\\
      &\le C\E\bigg[\int_0^T|\alpha_t(X_t) -  \alpha_t(\widehat X^k_t)|^2dt\bigg] + C\E\bigg[\int_0^T|\alpha_t(\widehat X_t^k) - \widehat \alpha_t^k( \widehat X^k_t)|^2dt\bigg]+ C {\delta_k}.
     \end{align*}    
  Thus, using Lispchitz-continuity of $\alpha$ and Gronwall's inequality yields
  \begin{align*}
         \E\Big[\sup_{t\in [0,T]} &|\widehat X^k_t - X_t|^2\Big] \\
        &\le  C\E\bigg[\int_0^T|\alpha_t(\widehat X_t^k) - \widehat \alpha_t^k(\widehat X^k_t)|^2dt\bigg]+ C {\delta_k}\\
        & = C\E^{\P^{\alpha^k}}\bigg[\int_0^T|\alpha_t(\widehat Y_t^k) - \widehat \alpha_t^k(\widehat Y^k_t)|^2dt\bigg]+ C {\delta_k} \\
        & \le C\E^{\P^{\widehat \alpha^k}}\bigg[\int_0^T|\alpha_t(\widehat Y_t^k) - \alpha_t(Y_t)|^2dt\bigg]+ C\E^{\P^{\alpha^k}}\bigg[\int_0^T|\alpha_t(Y_t) - \widehat \alpha_t^k(\widehat Y^k_t)|^2dt\bigg] + C {\delta_k} \\
        &\le C\E\Big[\Big( \frac{d\mathbb P^{\alpha^k}}{d\mathbb P} \Big)^2\Big]^{\frac12}\E\bigg[\int_0^T|\widehat Y^k_t - Y_t|^2dt\bigg]^{\frac12} + C {\delta_k}\\
        &\le  C {\delta_k},
  \end{align*}
  where we used \eqref{eq estimate Yk to Y rate} and the boundedness of $A$.
  With this estimation at hand, the bound on the controls follows by essentially the same arguments as above:
  \begin{align*}
    \E\bigg[\int_0^T|\alpha_t(X_t) -\widehat \alpha^k_t(\widehat X^k_t)|^2dt\bigg] &\le C\E\bigg[\int_0^T|\alpha_t(X_t) - \alpha_t(\widehat X^k_t)|^2dt\bigg] + \E^{\P^{\alpha^k}}\bigg[\int_0^T|\alpha_t(\widehat Y^k_t) - \widehat \alpha_t^k( \widehat Y^k_t)|^2dt\bigg]\\
    &\le C\E\bigg[\int_0^T|X_t - \widehat X^k_t|^2dt\bigg] + C\E^{\P^{\alpha^k}}\bigg[\int_0^T|Y_t -\widehat Y^k_t|^2dt\bigg] \\
    &\quad + C\E^{\P^{\alpha^k}}\bigg[\int_0^T|\alpha_t(Y_t) -\widehat \alpha_t^k(\widehat Y^k_t)|^2dt \bigg]\\
    &\le C {\delta_k},
  \end{align*}
  which, since $\delta_k = T/k$, completes the proof.
  \end{proof}

\begin{remark}\label{remark:alternativeconditions}
    We emphasize on the fact that the regularity conditions on $\alpha$ are only used in \eqref{label:alpharegularity}. Alternative conditions on the \emph{process} can also achieve the same rate, given the fact that $\alpha$ can be identified with a Lipschitz function of $\cZ$, the BSDE solution. Indeed, under strong convexity, the optimal control is $\P\otimes dt$ unique, and $\alpha$ can be identified with $\hat{a}(X_t, \cZ_t)$ where $\hat{a}$ is the minimizer function of the reduced Hamiltonian $h_0$. The same rate can then be achieved if $\hat{a}$ is Lipschitz 
    (as e.g.\ under the conditions of \cite[Lemma 3.3]{cardelbook1}) and $\cZ$ satisfies $$\E\big[|\cZ_t - \cZ_s|^p\big] \leq C|t-s|^{p/2},$$
    which can be achieved, for example, in \cite{hu2011malliavin}.
\end{remark}

\subsection{A possible simulation method}
\label{sec:simulation}
In light of the results of this section, we take a moment to comment on a possible simulation method.

It is interesting to notice that computing solutions of difference equations can be done via simple recursion methods.
Therefore, combining Theorem \ref{thm:compactness} and Proposition \ref{prop:Hamilton} suggests a simulation algorithm for MFGs.
To make this idea more precise, observe that the difference equation  \eqref{eq:BSDeltaE} is equivalent to 
\begin{equation}
  \label{eq:BSDeltaE_iterate}
    \cY_{t_i} = \cY_{t_{i+1}}  + H(Y_{t_i},  m_{t_i}, \cZ_{t_i})\delta - \cZ_{t_i}(Z^k_{t_{i+1}} - Z^k_{t_i}) -\big(C_{t_{i+1}}- C_{t_i} \big) \quad \P\text{-a.s. for all}\quad i=0,\dots, k-1
  \end{equation}
  with $\cY_{t_k} = G(Y_{t_k}, m_{t_k})$ and that any solution $(\alpha_{t_i},m_{t_i})_{i=0,\dots,k}$ of the MFG would satisfy $\alpha_{t_i} \in \argmin_{a\in A}h(Y_{t_i}, a, m_{t_i}, \cZ_{t_i})$ and $m_{t_i} = \P^{\alpha,m}\circ Y_{t_i}^{-1}$.
To calculate $\cZ_{t_i}$ (and thus $\alpha_{t_i}$), we take conditional expectation with respect to $\cF^{Z^k}_{t_i}$ in \eqref{eq:BSDeltaE_iterate} to get $$\cY_{t_i} = \E[\cY_{t_{i+1}}|\cF^{Z^k}_{t_i}]+ H(Y_{t_i},  m_{t_i}, \cZ_{t_i})\delta.$$
Next,  multiplying \eqref{eq:BSDeltaE_iterate} by $Z^k_{t_{i+1}} - Z^k_{t_i}$, and then taking expectation, it follows by orthogonality of $C$ and independence of the increments of $Z^k$ that
\begin{equation}
\label{eq:BSDeltaE.Zsol}
  \cZ_{t_i} = \E[\cY_{t_{i+1}}(Z^k_{t_{i+1}} - Z^k_{t_i})|\cF^{Z^k}_{t_i}]/\delta
\end{equation}
 and consequently, $C_{t_{i+1}} - C_{t_i} = \cY_{t_{i+1}} - \E[\cY_{t_{i+1}} \mid \cF^{Z^k}_{t_i}] - \cZ_{t_{i+1}}(Z^k_{t_{i+1}} - Z^k_{t_i})$.
This suggests a simulation algorithm for the mean field equilibrium (at least when the Hamiltonian admits a unique maximizer).
\begin{enumerate}
  \item Pick a flow of measure $m^0\in \cP(\R^d)^k$ to initialize the algorithm. 
  \item Let $\cY_{t_k} = G(Y_{t_k}, m_{t_k}^0)$ and $\cZ_{t_k} = 0$.
  \item For $i=k-1,\dots,0$, recursively compute
  \begin{itemize}
    \item $\cZ_{t_i} = \frac{1}{\delta}\E[\cY_{t_{i+1}}(Z^k_{t_{i+1}} - Z^k_{t_i})|\cF^{Z^k}_{t_i}]$,
    \item $\cY_{t_i} = \E[\cY_{t_{i+1}}| \cF^{Z^k}_{t_i}] + \delta \inf_{a\in A}h(Y_{t_i}, m_{t_i}^0, a, \cZ_{t_i})$.
  \end{itemize}
      \item Update $m^0$ as $m^0_{t_i} = \P^{\alpha,m^0}\circ Y^{-1}_{t_i}$ and repeat from step (2) until consecutive measures $m^0$ become close enough.
\end{enumerate}
This simulation method is interesting insofar that it does not involve solving an infinite dimensional optimization problem (but only optimization of $H$ over $A$) or solving coupled systems of forward-backward SDEs of McKean-Vlasov type, or PDE systems as current state of the art methods do. We refer to \cite{Day-Lau24,Ach-Cap10,Laur21} and references therein for surveys of results and methods on simulation of MFGs. 

The above suggested algorithm was applied (albeit in a simple linear quadratic case) to an optimal execution problem in finance in \cite{Tangpi-Wang24FS}.
In our framework, Theorem \ref{thm:compactness} justifies convergence of this algorithm, and an explicit convergence rate is shown in Theorem \ref{thm:BSDE.rate.LL}.

\bibliographystyle{plainnat}
\bibliography{dispmonotone, jodi.bib}

\end{document}